\documentclass[a4paper,10pt]{amsart}

\usepackage[left=1.33in,right=1.33in,top=1.1in,bottom=1.1in,includehead,includefoot]{geometry}
\usepackage[T1]{fontenc}
\usepackage{lmodern,microtype}
\usepackage{amsthm,amsmath,amssymb,colonequals}
\usepackage{graphicx}
\usepackage{enumitem}
\usepackage[pdftitle={Ramsey numbers and monotone colorings},
            pdfauthor={M. Balko}]{hyperref}
\hypersetup{
unicode=true,
}

\newtheorem{theorem}{Theorem}
\newtheorem{lemma}[theorem]{Lemma}
\newtheorem{corollary}[theorem]{Corollary}

\newtheorem{problem}[theorem]{Problem}

\theoremstyle{remark}

\newtheorem*{claim*}{Claim}

\makeatletter
\let\old@setaddresses\@setaddresses
\def\@setaddresses{\bgroup\parindent 0pt\let\scshape\relax\old@setaddresses\egroup}
\makeatother

\DeclareMathOperator{\OR}{\overline{R}} 
\DeclareMathOperator{\ORS}{\overline{R}_{mon}} 
\DeclareMathOperator{\OT}{OT} 
\DeclareMathOperator{\tow}{tow} 
\DeclareMathOperator{\ES}{ES} 
\DeclareMathOperator{\R}{R} 

\title{Ramsey numbers and monotone colorings}

\author{Martin Balko}

\address[Martin Balko]{Department of Applied Mathematics and Institute for Theoretical Computer Science, Faculty of Mathematics and Physics, Charles University, Prague, Czech Republic}
\email{balko@kam.mff.cuni.cz}

\address[Martin Balko]{Department of Computer Science, Faculty of Natural Sciences, Ben-Gurion University of the Negev, Beer~Sheva, Israel}

\thanks{The project leading to this application has received funding from European Research Council (ERC)
under the European Unions Horizon 2020 research and innovation programme under grant agreement No.
678765.
The author was also supported by the grant 1452/15 from Israel Science Foundation and by the grant no.~18-13685Y of the Czech Science Foundation (GA\v{C}R).
The author also acknowledges the support of the Center for Foundations of Modern Computer Science (Charles University project UNCE/SCI/004).
}

\begin{document}

\begin{abstract}
For positive integers $N$ and $r \geq 2$, an \emph{$r$-monotone coloring} of $\binom{\{1,\dots,N\}}{r}$ is a 2-coloring by $-1$ and $+1$ that is monotone on the lexicographically ordered sequence of $r$-tuples of every $(r+1)$-tuple from~$\binom{\{1,\dots,N\}}{r+1}$.
Let $\ORS(n;r)$ be the minimum $N$ such that every $r$-monotone coloring of $\binom{\{1,\dots,N\}}{r}$ contains a monochromatic copy of $\binom{\{1,\dots,n\}}{r}$.

For every $r \geq 3$, it is known that $\ORS(n;r) \leq \tow_{r-1}(O(n))$, where $\tow_h(x)$ is the \emph{tower function} of height $h-1$ defined as $\tow_1(x)=x$ and $\tow_h(x) = 2^{\tow_{h-1}(x)}$ for $h \geq 2$.
The Erd\H{o}s--Szekeres Lemma and the Erd\H{o}s--Szekeres Theorem imply $\ORS(n;2)=(n-1)^2+1$ and $\ORS(n;3)=\binom{2n-4}{n-2}+1$, respectively.
It follows from a result of Eli\'{a}\v{s} and Matou\v{s}ek that  $\ORS(n;4)\geq\tow_3(\Omega(n))$.

We show that $\ORS(n;r)\geq \tow_{r-1}(\Omega(n))$ for every $r \geq 3$.
This, in particular, solves an open problem posed by Eli\'{a}\v{s} and Matou\v{s}ek and by Moshkovitz and Shapira. 
Using two geometric interpretations of monotone colorings, we show connections between estimating $\ORS(n;r)$ and two Ramsey-type problems that have been recently considered by several researchers.
Namely, we show connections with higher-order Erd\H{o}s--Szekeres theorems and with Ramsey-type problems  for order-type homogeneous sequences of points.

We also prove that the number of $r$-monotone colorings of $\binom{\{1,\dots,N\}}{r}$ is $2^{N^{r-1}/r^{\Theta(r)}}$ for $N \geq r \geq 3$, which generalizes the well-known fact that the number of simple arrangements of~$N$ pseudolines is $2^{\Theta(N^2)}$.
\end{abstract}

\maketitle

\section{Introduction}
\label{sec:introduction}

Let $r \ge 2$ be an integer.
An \emph{ordered $r$-uniform hypergraph} is a pair $\mathcal{H} = (H, \prec)$ consisting of an $r$-uniform hypergraph $H$ and a total ordering $\prec$ of the vertices of $H$.
Let $\mathcal{H}_1=(H_1,\prec_1)$ and $\mathcal{H}_2=(H_2,\prec_2)$ be two ordered $r$-uniform hypergraphs.
We say that $\mathcal{H}_1$ and $\mathcal{H}_2$ are \emph{isomorphic} if there is an isomorphism between  $H_1$ and $H_2$ that preserves the orders $\prec_1$ and~$\prec_2$.
The ordered hypergraph $\mathcal{H}_1$ is an \emph{ordered sub-hypergraph} of~$\mathcal{H}_2$ if $H_1$ is a sub-hypergraph of$~H_2$ and $\prec_1$ is a suborder of $\prec_2$.

For a positive integer $n$, we let $\mathcal{K}^r_n$ be the ordered complete $r$-uniform hypergraph on $n$ vertices.
That is, the edge set of $\mathcal{K}^r_n$ consists of all $r$-element subsets of the vertex set.
We also use $\mathcal{P}^r_n$ to denote the \emph{monotone $r$-uniform path} on $n$ vertices.
That is, $\mathcal{P}^r_n=(P_n^r,\prec)$ is an ordered $r$-uniform $n$-vertex hypergraph with edges formed by $r$-tuples of consecutive vertices in~$\prec$.

A \emph{coloring} $c$ of an ordered $r$-uniform hypergraph $\mathcal{H}$ is a function that assigns some element from a finite set $C$ of \emph{colors} to each edge of $\mathcal{H}$.
We say that $\mathcal{H}$ is \emph{monochromatic} in $c$ if all edges of~$\mathcal{H}$ receive the same color via $c$.
If $|C|=k$, then we call $c$ a \emph{$k$-coloring} of $\mathcal{H}$.

The \emph{ordered Ramsey number} $\OR(\mathcal{H})$ of an ordered $r$-uniform hypergraph $\mathcal{H}$ is the minimum positive integer $N$ such that for every 2-coloring $c$ of $\mathcal{K}^r_N$ there is a sub-hypergraph of $\mathcal{K}^r_N$ that is monochromatic in $c$ and isomorphic to $\mathcal{H}$.
It follows from Ramsey's theorem that ordered Ramsey numbers always exist and are finite.
There are examples of ordered graphs $\mathcal{G}=(G,\prec)$, for which ordered Ramsey numbers $\OR(\mathcal{G})$ differ significantly from the standard Ramsey numbers $\R(G)$.
For example, there are ordered matchings $\mathcal{M}=(M,\prec)$ on $n$ vertices for which $\R(M)$ is only linear in $n$, while $\OR(\mathcal{M})$ grows superpolynomially in~$n$~\cite{bckk15,cfls17}.

The motivation for studying the growth rate of the ordered Ramsey numbers $\OR(\mathcal{P}^r_n)$ of monotone $r$-uniform paths comes from the classical paper by Erd\H{o}s and Szekeres~\cite{erdSze35}.
In this paper, which was one of the starting points of both Ramsey theory and discrete geometry, Erd\H{o}s and Szekeres independently reproved Ramsey's Theorem and also proved two other important results in Ramsey theory, the \emph{Erd\H{o}s--Szekeres Theorem} about point sets in convex position and the \emph{Erd\H{o}s--Szekeres Lemma} on monotone subsequences.
The latter results states that for every $n \in \mathbb{N}$ there is a positive integer $N(n) = (n-1)^2+1$ such that every sequence of $N(n)$ numbers contains a nondecreasing or a nonincreasing subsequence of length $n$.
Moreover, the number $N(n)$ is minimum possible, as there are sequences of $(n-1)^2$ numbers without a monotone subsequence of length $n$.
It is easy to show that $N(n) \leq \OR(\mathcal{P}^2_n)$.
In fact, $N(n)=\OR(\mathcal{P}^2_n)=(n-1)^2+1$~\cite{msw15}.
The Erd\H{o}s--Szekeres Theorem states that for every $n \in \mathbb{N}$ there is a positive integer ${\rm ES}(n)$ such that every set of ${\rm ES}(n)$ points in the plane with no three collinear points contains $n$ points that are vertices of a convex $n$-gon.
This result is closely connected to the problem of estimating $\OR(\mathcal{P}^3_n)$.
Erd\H{o}s and Szekeres~\cite{erdSze35} showed ${\rm ES}(n) \leq \binom{2n-4}{n-2}+1$.
We can again rather easily show that ${\rm ES}(n) \leq \OR(\mathcal{P}^3_n)$.
The bound of Erd\H{o}s and Szekeres then follows from the fact $\OR(\mathcal{P}^3_n) = \binom{2n-4}{n-2}+1$ for every $n \geq 2$~\cite{fpss12,moshSha14}.
Moreover, several other interesting geometric applications of estimates on $\OR(\mathcal{P}^r_n)$ for $r \geq 3$ appeared, for example, variants of the Erd\H{o}s--Szekeres Theorem for convex bodies~\cite{fpss12} or the higher-order Erd\H{o}s--Szekeres theorems~\cite{eliMat13}.

Given this motivation, the ordered Ramsey numbers $\OR(\mathcal{P}^r_n)$ have been recently quite intensively studied~\cite{eliMat13,fpss12,msw15,moshSha14} and their growth rate is nowadays well understood.
For positive integers $n$ and $h$, let $\tow_h(n)$ be the \emph{tower function} of height $h-1$.
That is, $\tow_1(n) = n$ and $\tow_h(n) = 2^{\tow_{h-1}(n)}$ for every $h \geq 2$.
Moshkovitz and Shapira~\cite{moshSha14} showed that, for all positive integers $n$ and $r$ with $r \geq 3$, \begin{equation} 
\label{eq-RamseyNumberMonotonePath}
\OR(\mathcal{P}^r_{n+r-1}) = \tow_{r-1}((2-o(1))n).
\end{equation}

In fact, Moshkovitz and Shapira~\cite{moshSha14} proved $\OR(\mathcal{P}^r_{n+r-1})=\rho_r(n)+1$, where $\rho_r(n)$ is the number of \emph{line partitions of $n$ of order $r$} (see~\cite{moshSha14} for definitions).
For $r=3$, this gives the exact formula $\OR(\mathcal{P}^3_n) = \binom{2n-4}{n-2}+1$ and yields a new proof of the Erd\H{o}s--Szekeres Theorem~\cite{erdSze35}. 
Their coloring $c$ of~$\mathcal{K}^3_N=(K^3_N,\prec)$ that gives $\OR(\mathcal{P}^3_n) > \binom{2n-4}{n-2}$ satisfies the following \emph{transitivity property}: if $v_1 \prec v_2 \prec v_3 \prec v_4$ are vertices of $\mathcal{K}^3_N$ such that $c(\{v_1,v_2,v_3\}) = c(\{v_2,v_3,v_4\})$, then all triples from $\binom{\{v_1,v_2,v_3,v_4\}}{3}$ have the same color in $c$.

More generally, for an integer $r \geq 2$, a 2-coloring $c$ of $\mathcal{K}^r_N = (K^r_N,\prec)$ is called \emph{transitive} if for every $(r+1)$-tuple of vertices $\{v_1,\dots,v_{r+1}\}$ that satisfies $v_1 \prec \dots \prec v_{r+1}$ and $c(\{v_1,\dots,v_r\})=c(\{v_2,\dots,v_{r+1}\})$ it holds that all $r$-tuples from $\binom{\{v_1,\dots,v_{r+1}\}}{r}$ have the same color in $c$.
For an ordered hypergraph $\mathcal{H}$, let $\OR_{trans}(\mathcal{H})$ be the number $\OR(\mathcal{H})$ restricted to transitive 2-colorings.
That is, $\OR_{trans}(\mathcal{H})$ is the minimum positive integer $N$ such that for every transitive 2-coloring $c$ of $\mathcal{K}^r_N$ there is an ordered sub-hypergraph of~$\mathcal{K}^r_N$ that is monochromatic in $c$ and isomorphic to $\mathcal{H}$.

Note that $\OR_{trans}(\mathcal{P}^r_n) = \OR_{trans}(\mathcal{K}^r_n)$ for all positive integers $n$ and $r \geq 2$.
We also remark that $\OR_{trans}(\mathcal{P}^r_n) < \OR(\mathcal{K}^r_n)$ for every $r \geq 2$ and every sufficiently large $n$.
For example, $\OR_{trans}(\mathcal{P}^2_n) = (n-1)^2+1$~\cite{msw15}, while $\OR(\mathcal{K}^r_n)$ equals the standard Ramsey number $\R(K^r_n)$ of the complete $r$-uniform hypergraph on $n$ vertices and thus $\OR(\mathcal{K}^2_n)$ grows exponentially in $n$~\cite{erdos47}.

Perhaps surprisingly, the colorings of $\mathcal{K}^r_N$, which were found by Moshkovitz and Shapira~\cite{moshSha14} and which give $\OR(\mathcal{P}^r_{n+r-1})>\rho_r(n)$, are not transitive for $r>3$.
Thus it is natural to ask the following question.

\begin{problem}{\cite{eliMat13,moshSha14}}
\label{prob-transitivity}
What is the growth rate of $\OR_{trans}(\mathcal{P}^r_n)$?
\end{problem}

Problem~\ref{prob-transitivity} was considered by Eli\'{a}\v{s} and Matou\v{s}ek~\cite{eliMat13}, who asked for better lower bounds on $\OR_{trans}(\mathcal{P}^r_n)$.
Moshkovitz and Shapira~\cite{moshSha14} note that it might be very well possible that bounds comparable with the bounds for $\OR(\mathcal{P}^r_n)$ hold also for $\OR_{trans}(\mathcal{P}^r_n)$.
They also mention a problem of deciding whether $\OR(\mathcal{P}^r_n) = \OR_{trans}(\mathcal{P}^r_n)$ for all $n$ and $r$.

Clearly, $\OR_{trans}(\mathcal{P}^r_n) \leq \OR(\mathcal{P}^r_n)$ and, by~\eqref{eq-RamseyNumberMonotonePath}, $\OR_{trans}(\mathcal{P}^r_n)$ grows at most as a tower of height $r-2$.
This was also shown by Eli\'{a}\v{s} and Matou\v{s}ek~\cite{eliMat13}, who also proved $\OR_{trans}(\mathcal{P}^4_n) = \tow_3(\Theta(n))$.
Thus Problem~\ref{prob-transitivity} is settled for $r \leq 4$.
We are not aware of any other lower bound on $\OR_{trans}(\mathcal{P}^r_n)$.

In this paper, we settle Problem~\ref{prob-transitivity} by constructing,  for all $n$ and $r$ with $r \geq 3$, transitive colorings $c_r$ of~$\mathcal{K}^r_N$ with no monochromatic copy of $\mathcal{P}^r_{2n+r-1}$, where $N \geq \tow_{r-1}((1-o(1))n)$.
In fact, we show that the colorings $c_r$ satisfy so-called \emph{monotonicity property}, which is much more restrictive than the transitivity property and which admits several geometric interpretations.

\subsection{Monotone colorings}
\label{subsec:signotopes}

For a positive integer $n$, we write $[n]$ to denote the set $\{1,\dots,n\}$.
Let $S$ be a sequence of $n$ elements from some set.
For a subset $\{i_1,\dots,i_k\}$ of $[n]$, we use $S^{(i_1,\dots,i_k)}$ to denote the subsequence of $S$ obtained by deleting all elements from~$S$ that are at position $i_j$ for some $j \in [k]$.

Let $r \geq 2$ be an integer.
A 2-coloring $c$ of $\mathcal{K}^r_N=(K^r_N,\prec)$ is called an \emph{$r$-monotone coloring} of~$\mathcal{K}^r_N$ if it assigns $-1$ or $+1$ to every edge of $\mathcal{K}^r_N$ such that the following \emph{monotonicity property} is satisfied: for every sequence $S$ of $r+1$ vertices of $\mathcal{K}^r_N$ ordered by $\prec$ and all integers $i,j,k$ with $1 \leq i < j < k \leq r+1$,  we have $c(S^{(k)}) \leq c(S^{(j)}) \leq c(S^{(i)})$ or $c(S^{(k)}) \geq c(S^{(j)}) \geq c(S^{(i)})$.
In other words, the monotonicity condition says that there is at most one change of a sign in the sequence $(c(S^{(r+1)}), \dots, c(S^{(1)}))$.
When referring to a 2-coloring that is $r$-monotone for some $r \geq 2$, we sometimes use the term \emph{monotone}.
We also abbreviate $-1$ and $+1$ by $-$ and $+$, respectively.

Note that every $r$-monotone coloring of $\mathcal{K}^r_N$ is a transitive 2-coloring of $\mathcal{K}^r_N$.
For $r=2$, transitive and $2$-monotone colorings coincide.
However, for $r \geq 3$, the monotonicity property is much more restrictive than the transitivity property, as $\mathcal{K}^r_{r+1}$ admits $2^r+2$ transitive and only $2r+2$ $r$-monotone colorings.
An example of a transitive 2-coloring of $\mathcal{K}^3_4$, which is not $3$-monotone, is a function $c$ with $(c(\{1,2,3\}),c(\{1,2,4\}),c(\{1,3,4\}),c(\{2,3,4\}))=(-,+,-,+)$.

The notion of monotone colorings has been considered by several researchers~\cite{felWei01,miya17,zieg93} under different names.
In some sense, monotone colorings can be viewed as more natural than transitive colorings, as they admit various geometric interpretations; see Subsections~\ref{subsec-higherOrderES} and~\ref{subsec-orderType} for examples.

\section{Our results}
\label{sec:ourResults}

A \emph{monotone Ramsey number} $\ORS(\mathcal{H})$ of an ordered $r$-uniform hypergraph $\mathcal{H}$ is the minimum positive integer $N$ such that for every $r$-monotone coloring $c$ of $\mathcal{K}^r_N$ there is an ordered sub-hypergraph of $\mathcal{K}^r_N$ that is monochromatic in~$c$ and isomorphic to $\mathcal{H}$.

Since every monotone coloring is transitive, we get $\ORS(\mathcal{P}^r_n) \leq \OR_{trans}(\mathcal{P}^r_n)$ and also $\ORS(\mathcal{P}^r_n) = \ORS(\mathcal{K}^r_n)$ for all $n$ and $r \geq 2$.
It follows from~\eqref{eq-RamseyNumberMonotonePath} that $\ORS(\mathcal{P}^r_n)\leq \tow_{r-1}(O(n))$.
All known lower bounds on $\OR_{trans}(\mathcal{P}^r_n)$ are also true for $\ORS(\mathcal{P}^r_n)$.
That is, we have $\ORS(\mathcal{P}^2_n) = \OR_{trans}(\mathcal{P}^2_n) = \OR(\mathcal{P}^2_n) = (n-1)^2+1$~\cite{erdSze35}, $\ORS(\mathcal{P}^3_n) = \OR_{trans}(\mathcal{P}^3_n) = \OR(\mathcal{P}^3_n) =\binom{2n-4}{n-2}+1$~\cite{erdSze35}, and $\ORS(\mathcal{P}^4_n) = \tow_3(\Theta(n))$~\cite{eliMat13} for every $n \in \mathbb{N}$, as all the constructed transitive colorings in these results are actually monotone. 

As our first main result, we prove an asymptotically tight lower bound on $\ORS(\mathcal{P}^r_n)$ for $r \geq 3$.
Since $\ORS(\mathcal{P}^r_n) \leq \OR_{trans}(\mathcal{P}^r_n)$, this settles Problem~\ref{prob-transitivity}.

\begin{theorem}
\label{thm-lowerBound}
For positive integers $r$ and $n$ with $r \geq 3$, we have
\[\ORS(\mathcal{P}^r_{2n+r-1}) \geq \tow_{r-1}((1-o(1))n).\]
\end{theorem}

For $r \in \{3,4\}$, the lower bounds from Theorem~\ref{thm-lowerBound} asymptotically match the lower bounds obtained from results of Erd\H{o}s and Szekeres~\cite{erdSze35} and Eli\'{a}\v{s} and Matou\v{s}ek~\cite{eliMat13}, respectively.
Our construction is closer to the construction of Moshkovitz and Shapira~\cite{moshSha14}, which they used to show the tight bound $\OR(\mathcal{P}^r_{n+r-1}) \geq \rho_r(n)+1$.

Our bounds on $\ORS(\mathcal{P}^r_n)$ do not match the upper bounds on $\OR(\mathcal{P}^r_n)$ exactly and thus deciding whether $\OR_{trans}(\mathcal{P}^r_n) = \OR(\mathcal{P}^r_n)$ for all $r$ and $n$ remains an interesting open problem.
It is even possible that $\ORS(\mathcal{P}^r_n) = \OR(\mathcal{P}^r_n)$  for all $r$ and $n$.

Despite having several natural geometric interpretations, the monotone colorings seem to be quite unexplored.
For example, we are not aware of any non-trivial estimate on the number of $r$-monotone colorings of $\mathcal{K}^r_n$ for $r>3$.
Here, we derive both upper and lower bounds for this number.
Note that the bounds are reasonably close together, even with respect to $r$.

\begin{theorem}
\label{thm-signotopesCount}
For integers $r \geq 3$ and $n \geq r$, the number $S_r(n)$ of $r$-monotone colorings of $\mathcal{K}^r_n$ satisfies
\[2^{n^{r-1}/r^{4r}} \leq S_r(n) \leq 2^{2^{r-2}n^{r-1}/(r-1)!}.\]
\end{theorem}

As we will see in Subsection~\ref{subsec-orderType}, Theorem~\ref{thm-signotopesCount} is a generalization of the well-known fact that the number of simple arrangements of $n$ pseudolines is $2^{\Theta(n^2)}$.
This fact follows from Theorem~\ref{thm-signotopesCount} by setting $r=3$.
However, the constants in the exponents in the bounds from Theorem~\ref{thm-signotopesCount} are not the best known.
Felsner and Valtr~\cite{felVal11} showed that the number of simple arrangements of $n$ pseudolines is at most $2^{0.657n^2}$, improving the previous bounds $2^{0.792n^2}$ by Knuth~\cite{knuth92} and $2^{0.697n^2}$ by Felsner~\cite{felsner97}.
Felsner and Valtr~\cite{felVal11} also proved the lower bound $2^{0.188n^2}$.
All these bounds apply also to $S_3(n)$.

In the rest of this section we use two geometric interpretations of monotone colorings to show connections between the problem of estimating $\ORS(\mathcal{P}^r_n)$ and some geometric Ramsey-type problems that have been recently studied.

We note that besides the following two geometric interpretations of monotone colorings, there is also a third one, which was discovered by Ziegler~\cite{zieg93}.
He showed that monotone colorings can be interpreted as extensions of the cyclic arrangement of hyperplanes with a pseudohyperplane.

\subsection{Higher-order Erd\H{o}s--Szekeres theorems}
\label{subsec-higherOrderES}

Very recently, Miyata~\cite{miya17} introduced a new geometric interpretation of $(k+2)$-monotone colorings for $k \in \mathbb{N}$, which are called \emph{degree-$k$ oriented matroids} in~\cite{miya17}.
This interpretation concerns \emph{$k$-intersecting pseudoconfigurations of points} (or \emph{$k$-pseudoconfigurations}, for short), which are formed by a pair $(P,L)$ satisfying the following conditions.
The set $P=\{p_1,\dots,p_n\}$ contains $n$ points in the Euclidean plane ordered by their increasing $x$-coordinates and the set $L$ is a collection of $x$-monotone Jordan arcs such that:
\begin{enumerate}[label=(\roman*)]
\item for every $l \in L$, there are at least $k+1$ points of $P$ lying on $l$,
\item for every $(k+1)$-tuple of distinct points of $P$, there is a unique curve $l$ from $L$ passing through each point of this $(k+1)$-tuple,
\item any two distinct curves from $L$ cross at most $k$ times.\footnote{We count all crossings, not only those contained in $P$.}
\end{enumerate}

This notion naturally generalizes the concept of \emph{generalized point sets}~\cite{goodmanPollack84} (sometimes called \emph{abstract order types}), which correspond to $1$-pseudocon\-figurations.
It also captures the essential combinatorial properties of configurations of points and graphs of polynomial functions, which is a setting considered by Eli\'{a}\v{s} and Matou\v{s}ek~\cite{eliMat13} in their study of higher-order Erd\H{o}s--Szekeres theorems.

\begin{figure}[ht]
\centering
\includegraphics{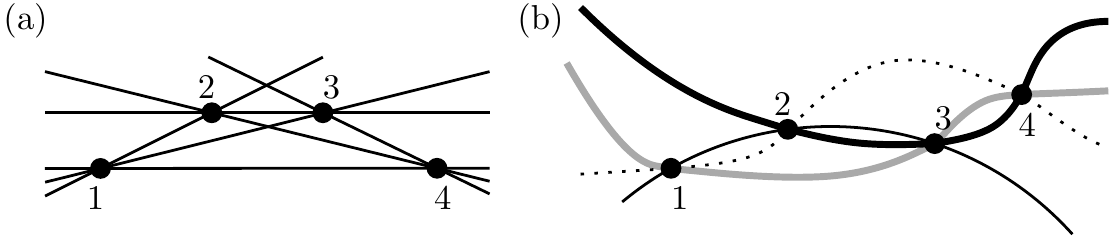}
\caption{Examples of simple $k$-pseudoconfigurations of four points for $k=1$ (part~(a)) and $k=2$ (part~(b)).
The sign function of the $1$-pseudoconfiguration maps each triple of points to $-$.
The sign function of the $2$-pseudoconfiguration assigns $+$ to the only $4$-tuple of points.}
\label{fig-pseudoconfiguration}
\end{figure}

A $k$-pseudoconfiguration $(P,L)$ of points is \emph{simple} if each curve from~$L$ passes through exactly $k+1$ points of $P$; see Figure~\ref{fig-pseudoconfiguration}.
If $(P,L)$ is simple, we let $l_{i_1,\dots,i_{k+1}}$ be the curve from~$L$ passing through points $p_{i_1},\dots,p_{i_{k+1}}$.
Each curve $l$ from~$L$ is a graph of a continuous function $f_l \colon \mathbb{R} \to \mathbb{R}$ and we let $l^- \colonequals \{(x,y) \in \mathbb{R}^2 \colon y < f_l(x)\}$. 
A \emph{sign function} of a simple $k$-pseudoconfiguration $(P,L)$ is a function $f \colon \binom{P}{k+2} \to \{-,+\}$ such that, given $\{i_1,\dots,i_{k+2}\} \in \binom{P}{k+2}$ with $i_1 < \dots < i_{k+2}$, $f(p_{i_1},\dots,p_{i_{k+2}})=-$ if and only if $p_{i_{k+2}} \in l^-_{i_1,\dots,i_{k+1}}$.

Miyata~\cite{miya17} proved the following correspondence between $(k+2)$-monotone colorings of~$K^{k+2}_n$ and simple $k$-pseudoconfigurations of $n$ points.

\begin{theorem}{\cite{miya17}}
\label{thm-interPseudoconf}
For $k,n \in \mathbb{N}$, there is a one-to-one correspondence between sign functions of simple $k$-pseudoconfigurations of $n$ points and $(k+2)$-monotone colorings of $\mathcal{K}^{k+2}_n$.
The monotone coloring corresponding to a $k$-pseudoconfiguration $\mathcal{P}$ is the sign function of $\mathcal{P}$.
\end{theorem}

A subset $S$ of $P$ is \emph{$(k+1)$st order monotone} if the sign function of $(P,L)$ attains only $-$ or only $+$ value on all of $(k+2)$-tuples of $S$. 
Theorem~\ref{thm-interPseudoconf} immediately gives the following corollary.

\begin{corollary}
\label{cor-signotopeES}
For all positive integers $k$ and $n$, the number $\ORS(\mathcal{P}^{k+2}_n)$ is the minimum positive integer $N$ such that every simple $k$-pseudoconfiguration of $N$ points contains a $(k+1)$st order monotone subset of size $n$.
\end{corollary}

Generalizing the Erd\H{o}s--Szekeres Theorem~\cite{erdSze35} to higher orders, Eli\'{a}\v{s} and Matou\v{s}ek~\cite{eliMat13} introduced the following more restrictive setting in which, for every $l \in L$, $f_l$ is a function whose $(k+1)$st derivative is everywhere non-positive or everywhere non-negative.
A planar point set $P$ is in \emph{$(k+1)$-general position} if no $k+2$ points of $P$ lie on the graph of a polynomial of degree at most $k$.
By Newton's interpolation, every $(k+1)$-tuple of points from $P$ determines a unique polynomial of degree at most $k$ whose graph contains this $(k+1)$-tuple and thus $P$ determines a simple $k$-pseudoconfiguration.
Thus, in this setting, we can consider $(k+1)$st order monotonicity with respect to the graphs of the polynomials of degree at most $k$.
Let $\ES_{k+1}(n)$ be the smallest positive integer $N$ such that every set of $N$ points in $(k+1)$-general position contains a $(k+1)$st order monotone subset of size~$n$.

By Corollary~\ref{cor-signotopeES}, we have $\ES_{k+1}(n) \leq \ORS(\mathcal{P}^{k+2}_n)$ for all positive integers $k$ and $n$.
It is known that this inequality is tight for $k =1$~\cite{erdSze35}.
Eli\'{a}\v{s} and Matou\v{s}ek~\cite{eliMat13} showed that $\ES_3(n)=\tow_3(\Theta(n))$ and thus $\ES_3(n)$ and $\ORS(\mathcal{P}^4_n)$ have asymptotically the same growth rate.
They also asked about the growth rate of $\ES_{k+1}(n)$ for $k > 2$.
A related interesting open question is whether $\ES_{k+1}(n)$ and $\ORS(\mathcal{P}^{k+2}_n)$ are the same, at least asymptotically.

By Corollary~\ref{cor-signotopeES}, it suffices to show that the extremal configurations for $\ORS(\mathcal{P}^{k+2}_n)$ can be `realized' by graphs of polynomial functions of degree at most $k$.
It is possible that the configurations obtained in the proof of Theorem~\ref{thm-lowerBound} admit such realizations, which would solve the open problem of Eli\'{a}\v{s} and Matou\v{s}ek about the growth rate of $\ES_{k+1}(n)$.
We hope to discuss this direction in future work.

\subsection{Arrangements of pseudohyperplanes and order-type homogeneous point sets}
\label{subsec-orderType}

Felsner and Weil~\cite{felWei01} showed that, for every $r \geq 3$, there is a one-to-one correspondence between $r$-monotone colorings of $\mathcal{K}^r_n$, which they call \emph{$r$-signotopes}, and arrangements of $n$ pseudohyperplanes in $\mathbb{R}^{r-1}$ that admit `sweeping'.

For an integer $d \geq 2$, a \emph{pseudohyperplane} $H$ in $\mathbb{R}^d$ is a homeomorph of a hyperplane in $\mathbb{R}^d$ such that the two connected components of $\mathbb{R}^d \setminus H$ are homeomorphic to an open $d$-dimensional ball.
Two pseudohyperplanes $H_1$ and $H_2$ \emph{cross}, if $H_i$ intersects both components of $\mathbb{R}^d \setminus H_{3-i}$ for every $i \in \{1,2\}$.
An \emph{arrangement of pseudohyperplanes} in $\mathbb{R}^d$ (or \emph{$d$-arrangement}, for short) consists of pseudohyperplanes $H_1,\dots,H_n$ in $\mathbb{R}^d$ such that any two pseudohyperplanes $H_i$ and $H_j$ intersect in a pseudohyperplane in $H_i \cong H_j \cong \mathbb{R}^{d-1}$ and they cross at their intersection.
Moreover, for every $j \in [n]$, the intersections $H_i \cap H_j$, where $i \in [n] \setminus \{j\}$, form an arrangement of pseudohyperplanes in $H_j \cong \mathbb{R}^{d-1}$.
A $d$-arrangement $\mathcal{A}$ is \emph{simple} if any $d+1$ pseudohyperplanes from $\mathcal{A}$ have an empty intersection.

We assume that every $d$-arrangement $\mathcal{A}$ of pseudohyperplanes $H_1,\dots,H_n$ is \emph{normal}, that is, $\mathcal{A}$ is simple and is embedded in $\mathbb{R}^d$ in the following normalized way.
Assume that $\mathcal{A}$ is embedded in the hypercube $[0,1]^d$ and, for $i \in [d-1]$, let $I_i$ be the $(d-i)$-dimensional subspace of $\mathbb{R}^d$ that  contains the side of $[0,1]^d$, which is obtained by setting the last $i$ coordinates to $0$.
We demand that $\mathcal{A} \cap I_i$ is a $(d-i)$-arrangement of $n$ pseudohyperplanes.
Moreover, the pseudohyperplanes in $\mathcal{A}$ are labeled by increasing first coordinate at their intersection with $I_{d-1}$.
The assumption that $\mathcal{A}$ is embedded in $[0,1]^d$ is only for convenience so that all intersections of $d$ pseudohyperplanes from $\mathcal{A}$ are contained in $[0,1]^d$.
The reader may consider ``spaces at infinity'' instead by defining $I_i$ to be the $(d-i)$-dimensional affine subspace obtained by setting the last $i$ coordinates to some sufficiently small number.

A \emph{sign function} of a normal $d$-arrangement $\mathcal{A}$ of $n$ pseudohyperplanes $H_1,\dots,H_n$ is a function $f \colon \binom{[n]}{d+1} \to \{-,+\}$ such that, for given $i_1<\dots<i_{d+1}$, $f(i_1,\dots,i_{d+1})=-$ if and only if the pseudoline $H_{i_3} \cap \cdots \cap H_{i_{d+1}}$, which  is oriented away from~$I_1$, intersects $H_{i_1}$ before $H_{i_2}$.

A normal $d$-arrangement $\mathcal{A}$ is called a \emph{$C_d$-arrangement} if the normal $(d-1)$-arrangement formed by $H \cap I_1$ for $H \in \mathcal{A}$ has no $+$ sign in its sign function.
We note that every normal arrangement of \emph{pseudolines} (that is, pseudohyperplanes in~$\mathbb{R}^2$) is a $C_2$-arrangement, but this is no longer true for $C_d$-arrangements with $d \geq 3$.
This is because, for $d \geq 3$, the arrangement induced by $\mathcal{A}$ is not uniquely determined, while for $C_d$-arrangements this arrangement must be the ``minimal one with respect to the sign function''.
An example of a $C_2$-arrangement can be found in Figure~\ref{fig-CdArrangement}.

\begin{figure}[ht]
\centering
\includegraphics{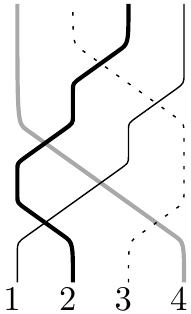}
\caption{A $C_2$-arrangement of four pseudolines.
Here, the sign function assigns $-$ to the triple $\{1,2,3\}$ and $+$ to the triple $\{2,3,4\}$.
}
\label{fig-CdArrangement}
\end{figure}

\begin{theorem}{\cite{felWei01}}
\label{thm-interHyperpl}
For $d \geq 2$ and $n \in \mathbb{N}$, there is a one-to-one correspondence between sign functions of $C_d$-arrangements of $n$ pseudohyperplanes in $\mathbb{R}^{d}$ and $(d+1)$-monotone colorings of~$\mathcal{K}^{d+1}_n$.
The monotone coloring corresponding to an arrangement $\mathcal{A}$ is the sign function of $\mathcal{A}$.
\end{theorem}

A subset $S$ of $\mathcal{A}$ is \emph{order-type homogeneous} if the sign function of $\mathcal{A}$ attains only $-$ or only $+$ values on all of $(d+1)$-tuples of pseudohyperplanes from~$S$. 
Theorem~\ref{thm-interHyperpl} gives the following corollary.

\begin{corollary}
\label{cor-signotopesOrderType}
For all positive integers $d \geq 2$ and $n$, the number $\ORS(\mathcal{P}^{d+1}_n)$ is the minimum positive integer $N$ such that every $C_d$-arrangement of $N$ pseudohyperplanes contains an order-type homogeneous subset of size $n$.
\end{corollary}

An \emph{orientation} of a $(d+1)$-tuple of points $(p_1,\dots,p_{d+1})$ with $p_i = (a_{i,1},\dots,a_{i,d}) \in \mathbb{R}^d$ is defined as
\[{\rm sgn} \det
\begin{pmatrix}
1 & 1 & 1 & 1 \\
a_{1,1} & a_{2,1} & \cdots & a_{d+1,1} \\
\vdots & \vdots & \cdots & \vdots \\
a_{1,d} & a_{2,d} & \cdots & a_{d+1,d}
\end{pmatrix}
.\]
A sequence of points from $\mathbb{R}^d$, $d \geq 2$, is \emph{order-type homogeneous} if all $(d+1)$-tuples of points from this sequence have the same orientation.
For positive integers $n$ and $d \geq 2$, let $\OT_d(n)$ be the minimum positive integer $N$ such that every sequence of $N$ points from $\mathbb{R}^d$ contains an order-type homogeneous subsequence with $n$ points.
Using geometric duality, the notion of order-type homogeneous sequence of points from $\mathbb{R}^d$ transcribes to sequences of hyperplanes in $\mathbb{R}^d$.
Thus $\OT_d(n)$ is also the minimum positive integer $N$ such that every sequence of $N$ hyperplanes in $\mathbb{R}^d$ contains an order-type homogeneous subsequence of size $n$.

The function $\OT_d(n)$ was considered by many researchers~\cite{barMatPor16,cfpss13,emrs14,suk14}.
Suk~\cite{suk14} showed that $\OT_d(n) \leq \tow_d(O(n))$.
The results of B\'{a}r\'{a}ny, Matou\v{s}ek, and P\'{o}r~\cite{barMatPor16} and Eli\'{a}\v{s}, Matou\v{s}ek, Rold\'{a}n-Pensado, and Safernov\'{a}~\cite{emrs14} give an asymptotically matching lower bound $\OT_d(n) \geq \tow_d(\Omega(n))$.
For $d \geq 3$, the arrangements obtained from their lower bound on $\OT_d(n)$ are not $C_d$-arrangements.
A natural problem is to decide whether one can obtain similar lower bounds on $\OT_d(n)$ when restricted to $C_d$-arrangements of hyperplanes.
Corollary~\ref{cor-signotopesOrderType} combined with Theorem~\ref{thm-lowerBound} suggests that this might be true, as we obtain such bounds for $C_d$-arrangements of pseudohyperplanes for every $d \geq 2$.

\section{Proof of Theorem~\ref{thm-lowerBound}}
\label{sec:proofLowerBound}

Here, for positive integers $n$ and $r$ with $r \geq 3$, we construct an $r$-monotone coloring $c_r$ of~$\mathcal{K}^r_N$ with no monochromatic copy of $\mathcal{P}^r_{2n+r-1}$ and with $N \geq \tow_{r-1}((1-o(1))n)$.
First, we describe the construction of $c_r$ and show that $c_r$ contains no long monochromatic monotone $r$-uniform paths.
Then we prove that the coloring $c_r$ satisfies the monotonicity property.

Let us start with a brief overview of the construction of the coloring $c_r$.
It is carried out iteratively with respect to $r$.
For every positive integer $n$, we will construct sets $F_r(n)$ with $r \geq 1$ such that $|F_1(n)|=2$, $|F_2(n)|=2n$, and $|F_r(n)| = 2^{|F_{r-1}(n)|/2}$ for $r \geq 3$.
The 2-coloring $c_r$ will have $F_r(n)$ as its vertex set.
We will have a partition of $F_r(n)$ into sets $F^-_r(n)$, $F^+_r(n)$, and a bijection $\sigma_r \colon F^-_r(n) \rightarrow F^+_r(n)$.
Elements $A,B \in F_r(n)$ will be called \emph{equivalent}, written $A \equiv_r B$, if $A=B$, $A=\sigma_r(B)$, or $B=\sigma_r(A)$.
We say that elements from $F^-_r(n)$ and $F^+_r(n)$ have \emph{type} $-$ and $+$, respectively.
We will also define two orders $<_r$ and $\prec_r$; $<_r$ will be a linear order on $F_r(n)$ and $\prec_r$ will be a linear order on equivalence classes under the equivalence relation $\equiv_r$.
In $<_r$, all elements of $F^-_r(n)$ will precede all elements of $F^+_r(n)$, and the bijection $\sigma_r$ will be order-reversing.
Moreover, if we regard $\prec_r$ as an ordering on $F_r^-(n)$ and on $F^+_r(n)$, we will have $(F^-_r(n),\prec_r) = (F^-_r(n),<_r)$, and hence $(F^+_r(n),\prec_r) = (F^+_r(n),>_r)$.
The color of an edge $e=\{A_1,\dots,A_r\}$ in~$c_r$, where $A_i \in F_r(n)$ and $A_1 <_r \cdots <_r A_r$, is then defined using an iterative application of a function $\gamma$ on consecutive vertices $A_i$ and $A_{i+1}$, where $\gamma(A,B)$ is the first element of $B$ in $\preceq_{r-1}$ on which $A$ and $B$ differ.
We apply $\gamma$ on $e$ until we reach a unique element of $F_1(n)$, which is set to be the color of~$e$.

Now, we proceed by describing the construction of $c_r$ in full detail.
Let $F_2^-(n) \colonequals \{(2n-i+1,i) \colon i \in [n]\} \subseteq [2n]^2$ and $F_2^+(n) \colonequals \{(i,2n-i+1) \colon i \in [n]\} \subseteq [2n]^2$.
We define a linear ordering $<_2$ on the disjoint union $F_2(n) \colonequals F_2^-(n) \cup F^+_2(n)$ by letting $(2n,1) <_2 (2n-1,2) <_2 \dots <_2 (1,2n)$\footnote{Alternatively, one might define $F_2(n)=[2n]$. However, we use this definition as it is more similar to the approach of Moshkovitz and Shapira.}.
Note that $N_2 \colonequals |F_2(n)|=2n$.
For convenience, we define $F^-_1(n) \colonequals \{-\}$, $F^+_1(n) \colonequals \{+\}$, $F_1(n) \colonequals \{-,+\}$, and $- <_1 +$.

Let $\sigma_2 \colon F_2^-(n) \to F_2^+(n)$ be the one-to-one correspondence that maps $(2n-i+1,i)$ to $(i,2n-i+1)$.
Two elements $A$ and $B$ from $F_2(n)$ are \emph{equivalent}, written $A \equiv_2 B$, if $A=B$, $A=\sigma_2(B)$, or $B=\sigma_2(A)$.
We order the equivalence classes of $F_2(n)$ under $\equiv_2$ by a linear order $\preceq_2$ by identifying each $A$ from $F_2^-(n)$ with $\sigma_2(A)$ and by letting $\prec_2$ be the ordering $<_2$ on $F_2^-(n)$.
Slightly abusing the notation, we sometimes consider $\preceq_2$ as a linear order on $F_2(n)$.
Then two equivalent elements of $F_2(n)$ are considered equal in $\preceq_2$.
For $r=1$, we let $\sigma_1(-)=+$ and $- \equiv_1 +$.

Let $r \ge 3$ be a positive integer and assume we have constructed $F_{r-1}(n)$.
Let $F_r(n)$ be the collection of sets $X$ such that $X$ contains exactly one set from each equivalence class of $\equiv_{r-1}$ on $F_{r-1}(n)$.
Observe that $N_r \colonequals |F_r(n)| = 2^{N_{r-1}/2}$ and that no two sets from $F_r(n)$ are comparable in $\subseteq$.
Also note that the minimum and the maximum element of $F_{r-1}(n)$ in $<_{r-1}$ are equivalent and thus $X$ contains exactly one of them.

We let $F^-_r(n)$ and $F^+_r(n)$ be the subsets of $F_r(n)$ consisting of sets that contain the minimum and the maximum element of $F_{r-1}(n)$ in $<_{r-1}$, respectively.
Since every element of $F_r(n)$ contains either the minimal or the maximal element of $F_{r-1}(n)$ in $<_{r-1}$, the sets $F^-_r(n)$ and $F^+_r(n)$ partition $F_r(n)$.
We say that sets from $F_r^-(n)$ and $F^+_r(n)$ have \emph{type} $-$ and $+$, respectively. 
An example for $r=3=n$ can be found in Figure~\ref{fig-F3}.

Let $A$ and $B$ be distinct sets from $F_r(n)$ for $r \geq 3$.
We let $\gamma(A,B)$ be the element from $B \cap E$, where $E$ is the first equivalence class of $(F_{r-1}(n))_{\equiv_{r-1}}$ in $\prec_{r-1}$ on which $A$ and $B$ differ.
We define the total order $<_r$ on $F_r(n)$ by letting $A<_rB$ if $\gamma(A,B) \in F_{r-1}^+(n)$.
Observe that $\gamma(A,B) \in F_{r-1}^+(n)$ if and only if $\gamma(B,A) \in F_{r-1}^-(n)$ and thus $<_r$ is indeed a total order.
For $r=2$, if $A=(a_1,a_2)$ and $B=(b_1,b_2)$ are distinct elements from $F_2(n)$, then we let $\gamma(A,B)=-$ if $a_1<b_1$ and, similarly, $\gamma(A,B)=+$ if $a_1>b_1$, where $<$ is the standard ordering of $\mathbb{R}$.
Note that, for $A,B \in F_2(n)$, $\gamma(A,B)=-$ if and only if $A >_2 B$.

We define the mapping $\sigma_r \colon F^-_r(n) \to F_r^+(n)$ by letting 
\[\sigma_r(\{A_1,\dots,A_{N_{r-1}/2}\}) \colonequals \{\sigma_{r-1}(A_1),\dots,\sigma_{r-1}(A_{N_{r-1}/2})\}.\]
Note that $\sigma_r$ is a one-to-one correspondence.
Two elements $A$ and $B$ from $F_r(n)$ are \emph{equivalent}, written $A \equiv_r B$, if $A=B$, $A = \sigma_r(B)$, or $B=\sigma_r(A)$.
We again order the equivalence classes of $F_r(n)$ under $\equiv_r$ by a linear order $\preceq_r$ that is obtained by identifying each $A$ from $F_r^-(n)$ with $\sigma_r(A)$ and by letting $\prec_r$ be the ordering $<_r$ on $F_r^-(n)$.
Again, slightly abusing the notation, we sometimes consider $\preceq_r$ as a linear order on the set $F_r(n)$.
Thus two equivalent elements from $F_r(n)$ are the same in $\preceq_r$, $(F_r^-(n),<_r)=(F_r^-(n),\prec_r)$, and $(F_r^+(n),>_r)=(F_r^+(n),\prec_r)$.

\begin{figure}[ht]
\centering
\includegraphics{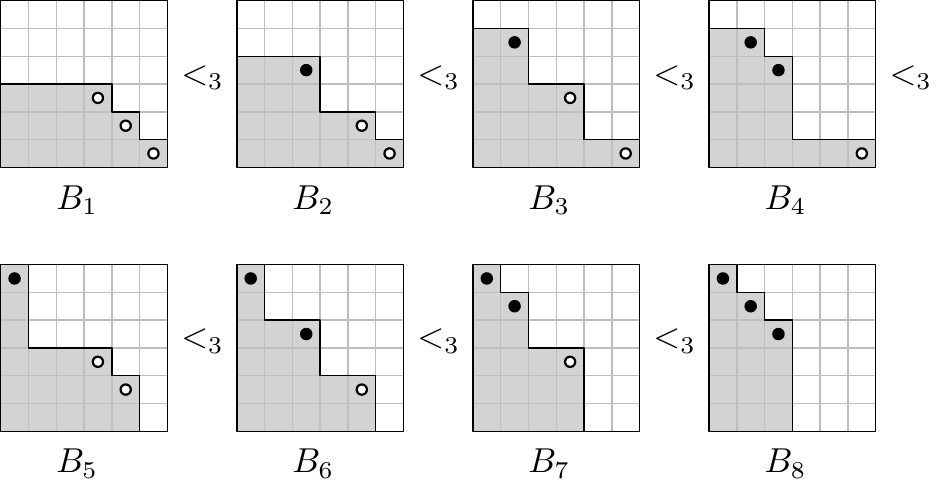}
\caption{Elements $B_1 <_3 \dots <_3 B_8$ of the set $F_3(3)$.
Each element $(i,j) \in \{1,\dots,6\}^2$ is represented by an entry on the $(7-i)$th row on the $j$th column of the corresponding matrix.
Thus, in particular, the elements of $F_2(3)$ form the diagonal.
The elements from $F_2^-(3)$ are denoted by empty circles, the elements from $F_2^+(3)$ by full circles.
The four sets on the first line have type $-$ and the four sets on the second line have type $+$. 
We have $B_i \equiv_3 B_{9-i}$ for every $i \in \{1,\dots,4\}$.
To illustrate the coloring $c_3$, we have $\gamma(B_1,B_2) = (3,4)$ and $\gamma(B_2,B_3) = (2,5)$, so $c_3(\{B_1,B_2,B_3\}) = +$.}
\label{fig-F3}
\end{figure}

For integers $k,r \geq 2$ and a sequence $(B_1,\dots,B_k)$ of sets from $F_r(n)$ in which any two consecutive terms are distinct, we use $\Gamma(B_1,\dots,B_k)$ to denote the sequence $(\gamma(B_1,B_2),\dots,\allowbreak\gamma(B_{k-1},B_k))$ of $k-1$ sets from $F_{r-1}(n)$.
Observe that, if $r \geq 3$, the definition of $\gamma$ guarantees that any two consecutive terms of $\Gamma(B_1,\dots,B_k)$ are distinct and thus we can apply the function $\Gamma$ on $F_{r-1}(n)$.
Applying $\Gamma$ to $(B_1,\dots,B_k)$ iteratively $i$ times, for some $i$ with $1 \leq i \leq \min\{k-1,r-1\}$, results in a sequence $\Gamma^i (B_1,\dots,B_k) \colonequals \Gamma(\Gamma(\cdots \Gamma(B_1,\dots,B_k) \cdots))$ of $k-i$ elements from $F_{r-i}(n)$.
For convenience, we set $\Gamma^0(B_1,\dots,B_k) \colonequals (B_1,\dots,B_k)$.

Letting $\mathcal{K}_{N_r}^r$ be the ordered complete $r$-uniform hypergraph with the vertex set $F_r(n)$ ordered by $<_r$, we color $\mathcal{K}^r_{N_r}$ with a 2-coloring $c_r$ by letting $c_r(\{A_1,\dots,A_r\}) \colonequals \Gamma^{r-1}(A_1,\dots,A_r)$ for every edge $\{A_1,\dots,A_r\}$ of $\mathcal{K}_{N_r}^r$ with $A_1 <_r \dots <_r A_r$.

\begin{lemma}
\label{lem-noLongPath}
For all positive integers $n$ and $r$ with $r \geq 3$, there is no monochromatic copy of $\mathcal{P}^r_{2n+r-1}$ in $\mathcal{K}_{N_r}^r$ colored with~$c_r$.
\end{lemma}
\begin{proof}
Let $\mathcal{P}$ be a monochromatic copy of $\mathcal{P}^r_k$ in $c_r$ for some integer $k \geq r$.
Let $A_1 <_r \dots <_r A_k$ be vertices of~$\mathcal{P}$.
Let $a_1, \dots, a_{k-r+2}$ be the elements of $F_2(n)$ obtained by applying the function $\Gamma^{r-2}$ to sequences $(A_1,\dots,A_{r-1}),(A_2,\dots,A_r),\allowbreak\dots,(A_{k-r+2},\dots,A_k)$, respectively.
The color $c_r(\{A_i,\dots,A_{i+r-1}\})$ of each edge $\{A_i,\dots,A_{i+r-1}\}$ of $\mathcal{P}$ then equals $\gamma(a_i,a_{i+1})$.
Thus if all edges of $\mathcal{P}$ have color $-$ in~$c_r$, we obtain $a_1 >_2 \dots >_2 a_{k-r+2}$.
That is, the first coordinates of $a_1,\dots,a_{k-r+2}$ increase and we get $k \leq 2n+r-2$, as $a_1,\dots,a_{k-r+2} \in F_2(n) \subseteq [2n]^2$.
Similarly, if all edges of $\mathcal{P}$ have color $+$, then $a_1 <_2 \dots <_2 a_{k-r+2}$ and the second coordinates of $a_1,\dots,a_{k-r+2}$ increase, which again implies $k \leq 2n+r-2$.
\end{proof}

Note that if $r=3$, then $a_1,\dots,a_{k-1}$ all have type $+$, as $A_1<_r\cdots <_r A_k$.
Using this fact, we could eventually obtain a better estimate $\ORS(\mathcal{P}^3_{n+2}) \geq 2^n$. 
However, this is not optimal anyway, as we know that $\ORS(\mathcal{P}^3_n) = \binom{2n-4}{n-2}+1$.

It remains to show that the coloring $c_r$ satisfies the monotonicity pro\-perty.
In other words, we want to show that there is at most one change of a sign in $(c_r(S^{(r+1)}),\dots,c_r(S^{(1)}))$ for every sequence $S=(A_1,\dots,A_{r+1})$ of sets from $F_r(n)$ with $A_1 <_r \dots <_r A_{r+1}$.
We first prove two auxiliary results that hold for every $r \geq 2$.

\begin{lemma}
\label{lem-deletion}
For positive integers $n$ and $r$ with $r \geq 2$, let $(A,B,C)$ be a sequence of distinct sets from~$F_r(n)$.
For $r \geq 3$, $\gamma(A,C) = \min_{\prec_{r-1}}\{\gamma(A,B),\allowbreak\gamma(B,C)\}$ if $\gamma(A,B) \not\equiv_r \gamma(B,C)$ and $\gamma(A,B),\gamma(B,C) \prec_{r-1} \gamma(A,C)$ otherwise.
For $r=2$, $\gamma(A,C) \in \{\gamma(A,B),\gamma(B,C)\}$ if $\gamma(A,B) \neq \gamma(B,C)$ and $\gamma(A,C) =\gamma(A,B) = \gamma(B,C)$ otherwise.
\end{lemma}
\begin{proof}
First, we assume $r \geq 3$.
One of the following three cases occurs: $\gamma(A,B) \prec_{r-1} \gamma(B,C)$, $\gamma(B,C) \prec_{r-1} \gamma(A,B)$, or $\gamma(A,B)\equiv_{r-1} \gamma(B,C)$.
In the first case, the sets in $B$ are the same as the sets in~$C$ up to $\gamma(B,C)$ in $\prec_{r-1}$ while the sets in $A$ and $B$ differ already on $\gamma(A,B) \prec_{r-1} \gamma(B,C)$ in~$\prec_{r-1}$.
Thus $\gamma(A,C)=\gamma(A,B)$.
Similarly, we obtain $\gamma(A,C)=\gamma(B,C)$ in the second case.

If $\gamma(A,B) \equiv_{r-1} \gamma(B,C)$, then it follows from $\gamma(A,B)\neq \gamma(B,C)$ that either $\gamma(A,B)=\sigma_r(\gamma(B,C))$ or $\sigma_r(\gamma(A,B))=\gamma(B,C)$.
In particular, $\gamma(B,A)=\gamma(B,C)$.
The sets in $A$ and $C$ thus differ for the first time on a set that is larger then both $\gamma(A,B)$ and $\gamma(B,C)$ in $\prec_{r-1}$.

For $r=2$, let $A=(a_1,a_2)$, $B=(b_1,b_2)$, and $C=(c_1,c_2)$.
If $\gamma(A,B)=+=\gamma(B,C)$, then $a_1>b_1$ and $b_1>c_1$.
In particular, $a_1>c_1$ and $\gamma(A,C)=+$.
Analogously, if $\gamma(A,B)=-=\gamma(B,C)$, then $\gamma(A,C)=-$.
If $\gamma(A,B) \neq \gamma(B,C)$, then $\gamma(A,C) \in \{\gamma(A,B),\gamma(B,C)\}$, as $F_1(n)$ contains only the values $-$ and $+$. 
\end{proof}

Note that if $A <_r B <_r C$ or $A >_r B >_r C$, then $\gamma(A,B)$ and $\gamma(B,C)$ have the same type and thus $\gamma(A,B) \not\equiv_{r-1} \gamma(B,C)$ if $r \geq 3$.
For $r \geq 3$, it follows from Lemma~\ref{lem-deletion} that if $\gamma(A,B) \equiv_{r-1} \gamma(B,C)$, then $\gamma(A,B) <_{r-1} \gamma(A,C) <_{r-1} \gamma(B,C)$ or $\gamma(A,B) >_{r-1} \gamma(A,C) >_{r-1} \gamma(B,C)$.
This is because the lemma gives us $\gamma(A,B),\gamma(B,C) \prec_{r-1} \gamma(A,C)$, which together with the facts $\gamma(A,B) \neq \gamma(B,C)$ and $\gamma(A,B) \equiv_{r-1} \gamma(B,C)$ implies that $\gamma(A,B)$ and $\gamma(B,C)$ have different types and thus $\gamma(A,C)$ lies between them in $<_{r-1}$.
For $r=2$, it follows that $(\gamma(A,B),\gamma(A,C),\gamma(B,C))$ has at most one change of a sign.
Thus, for any distinct $A,B,C$ from $F_r(n)$ with $r \geq 2$, the sequence $(\gamma(A,B),\gamma(A,C),\gamma(B,C))$ is monotone in $\leq_{r-1}$.

\begin{lemma}
\label{lem-replacement}
For positive integers $n$ and $r$ with $r \geq 2$, let $A,B,A',B'$ be sets from $F_r(n)$ such that $A \neq B$.
\begin{enumerate}[label=(\roman*)]
\item\label{item-lem-repl1}Assume $A' \neq B$. If $A \leq_r A'$, then $\gamma(A,B) \geq_{r-1} \gamma(A',B)$. 
\item\label{item-lem-repl2}Assume $A \neq B'$. If $B \leq_r B'$, then $\gamma(A,B) \leq_{r-1} \gamma(A,B')$.
\end{enumerate}
\end{lemma}
\begin{proof}
We prove only part~\ref{item-lem-repl1}, as the proof of part~\ref{item-lem-repl2} is analogous.
It is easy to verify the statement for $r=2$ and thus we consider $r \geq 3$.
We can assume $A \neq A'$, as otherwise the statement is trivial.
There are three possibilities where to place $B$ with respect to~$A$ and $A'$ in~$<_r$.
If $A <_r A' <_r B$, then Lemma~\ref{lem-deletion} implies $\gamma(A,B) \preceq_{r-1} \gamma(A',B)$ and, since $\gamma(A,B),\gamma(A',B) \in F_{r-1}^+(n)$, we have $\gamma(A,B) \geq_{r-1} \gamma(A',B)$.
If $A <_r B <_r A'$, then $\gamma(A,B) \in F_{r-1}^+(n)$ and $\gamma(A',B) \in F_{r-1}^-(n)$ and we obtain $\gamma(A,B) \geq_{r-1} \gamma(A',B)$ immediately.
Finally, if $B <_r A <_r A'$, then Lemma~\ref{lem-deletion} implies $\gamma(B,A') \preceq_{r-1} \gamma(B,A)$.
Since $\gamma(B,A)$ and $\gamma(A,B)$ are equivalent and have distinct type, and the same is true for $\gamma(B,A')$ and $\gamma(A',B)$, we have $\gamma(A',B) \preceq_{r-1} \gamma(A,B)$.
Using the fact that $\gamma(A,B),\gamma(A',B) \in F_{r-1}^-(n)$, we again obtain $\gamma(A,B) \geq_{r-1} \gamma(A',B)$.
\end{proof}

Before stating the last auxiliary result, we first introduce some definitions.
For two sequences $S_1$ and $S_2$, we use $S_1 \cdot S_2$ to denote the concatenation of $S_1$ and $S_2$.
A \emph{profile} is a sequence of symbols $\leq$, $\geq$, and $=$, containing at least one of the symbols $\leq$ and $\geq$.
Let $O_l \colonequals (\leq,=,\leq,=,\dots)$ and $E_l \colonequals (=,\geq,=\geq,\dots)$ be two profiles of length $l \in \mathbb{N}$.
We say that a profile $P$ of length $l$ is \emph{odd} or \emph{even} if it can be obtained from $O_l$ or $E_l$, respectively, by changing some occurrences of $\leq$ and $\geq$ to $=$.
For two profiles $P_1$ and $P_2$ such that each is odd or even, if $P_1$ is odd and $P_2$ is even, then $P_1$ and $P_2$ have \emph{distinct parity}.
Otherwise we say that $P_1$ and $P_2$ have the \emph{same parity}.
The \emph{opposite profile $\overline{P}$} of a profile $P$ is the profile that is obtained from $P$ by replacing each term $\leq$ with $\geq$ and each term $\geq$ with $\leq$. 

For positive integers $n$, $r$, and $s \geq 2$, let $R=(B_1,\dots,B_s)$ be a sequence of $s$ sets from~$F_r(n)$ and let $P$ be a profile of length $s-1$.
We say that $P$ is a \emph{profile of $R$} if whenever $B_j <_r B_{j+1}$ or $B_j >_r B_{j+1}$, then the $j$th term of $P$ is $\leq$ or $\geq$, respectively, for every $j \in [s-1]$. 

\begin{lemma}
\label{lem-Monotone}
For positive integers $n$, $r$, and $s$ with $r \geq 3$ and $3 \leq s \leq r+1$, let $S\colonequals(A_1,\dots,A_s)$ be a sequence of $s$ sets from $F_r(n)$ with $A_1 <_r \dots <_r A_s$.
Then the sequence $H\colonequals(\Gamma^{s-2}(S^{(s)}),\dots,\Gamma^{s-2}(S^{(1)}))$ has either odd or even profile.
\end{lemma}
\begin{proof}
We recall that, for a sequence $S$ and a subset $\{i_1,\dots,i_k\}$ of $\{1,\dots,|S|\}$, we use $S^{(i_1,\dots,i_k)}$ to denote the subsequence of $S$ obtained by deleting all elements from~$S$ that are at position $i_j$ for some $j \in [k]$.
Also note that every sequence $(A_1,\dots,A_k)$ of elements from $F_r(n)$ satisfies $\Gamma^{k-1}(A_1,\dots,A_k) = \gamma(\Gamma^{k-2}(A_1,\dots,A_{k-1}),\Gamma^{k-2}(A_2,\dots,A_k))$. 
In particular, we have $\Gamma^{s-2}(S^{(s)})=\allowbreak\gamma(\Gamma^{s-3}(S^{(s-1,s)})$, $\Gamma^{s-3}(S^{(1,s)}))$, 
$\Gamma^{s-2}(S^{(1)})=\gamma(\Gamma^{s-3}(S^{(1,s)}),\allowbreak\Gamma^{s-3}(S^{(1,2)}))$, 
and $\Gamma^{s-2}(S^{(i)})=\gamma(\Gamma^{s-3}(S^{(i,s)}),\Gamma^{s-3}(S^{(i,1)}))$ for every $i$ with $2 \leq i \leq s-1$. 

We use $H_1$ to denote the sequence $(\Gamma^{s-3}(S^{(s-1,s)}),\dots,\Gamma^{s-3}(S^{(1,s)}))$ and $H_2$ to denote $(\Gamma^{s-3}(S^{(1,s)}),\dots,\Gamma^{s-3}(S^{(1,2)}))$.
Let $G_1 \colonequals (\Gamma^{s-3}(S^{(s-1,s)})) \cdot H_1$ and $G_2 \colonequals H_2 \cdot (\Gamma^{s-3}(S^{(1,2)}))$.
That is, $G_1$ is the sequence obtained from $H_1$ by doubling the first term and $G_2$ is the sequence obtained from $H_2$ by doubling the last term.
By the definition of the function~$\gamma$, for every $i \in [s]$, the $i$th term of $H$ equals $\gamma(X,Y)$, where $X$ is the $i$th term of~$G_1$ and $Y$ is the $i$th term of $G_2$.

We proceed by induction on $s \geq 3$ and, in each step of the induction, we construct a profile $p(H)$ such that $p(H)$ is a profile of $H$ and $p(H)$ is odd or even.
We start with the base case $s=3$.
We have $H=(\gamma(A_1,A_2),\gamma(A_1,A_3),\gamma(A_2,A_3))$, $H_1=(A_1,A_2)$, $G_1 = (A_1,A_1,A_2)$, $H_2=(A_2,A_3)$, and $G_2 = (A_2,A_3,A_3)$.
Since $A_1 <_r A_2 <_r A_3$, it follows from Lemma~\ref{lem-deletion} that 
$\gamma(A_1,A_2)=\gamma(A_1,A_3)>_{r-1}\gamma(A_2,A_3)$ if $\Gamma(S^{(3)})\prec\Gamma(S^{(1)})$ or
$\gamma(A_1,A_2)<_{r-1}\gamma(A_1,A_3)=\gamma(A_2,A_3)$ if $\Gamma(S^{(1)})\prec_{r-1}\Gamma(S^{(3)})$. 
We thus choose $p(H)$ to be the even profile $(=,\geq)$ or the odd profile $(\leq,=)$, respectively.
We also set $p(H_1) \colonequals (\leq)$, $p(H_2) \colonequals (\leq)$, $p(G_1) \colonequals (=,\leq)$, and $p(G_2) \colonequals (\leq,=)$.
Observe that if $\Gamma(S^{(3)})\prec_{r-1} \Gamma(S^{(1)})$, then $p(H)$ is the profile $\overline{p(G_1)}$ and if $\Gamma(S^{(1)}) \prec_{r-1} \Gamma(S^{(3)})$, then $p(H)$ is the profile $p(G_2)$.

Let $R$ be a sequence of length $k$ with the profile $p(R)$ assigned.
We recall that the length of $p(R)$ is $k-1$.
We let $i_1(R)$ be the largest $i \in [k]$ such that the first $i-1$ terms of $p(R)$ are all $=$.
Similarly, we let $i_2(R)$ be the smallest $j \in [k]$ such that the last $k-j$ terms of $p(R)$ are all $=$.
In other words, $i_1(R)$ is the smallest $i$ with $1 \leq i \leq k$ such that the $i$th term of $p(R)$ is not $=$, and $i_2(R)$ is the smallest $i$ with $1 \leq i \leq k$ such that for every $j$ with $i \leq j \leq k-1$, the $j$th term of $p(R)$ is $=$.
Note that $i_2(R) \geq i_1(R)+1$.
In the case $s=3$, it is easy to check that $p(H_1)$ and $p(H_2)$ have the same parity and $i_1(G_1) = i_2(G_2)$.

For the induction step, we assume that $s \geq 4$.
We first express each of the sequences $H_1$ and $H_2$ as a result of applying $\gamma$ to two sequences, similarly as we have expressed $H$ using $G_1$ and $G_2$.
Let $H_{1,1} \colonequals (\Gamma^{s-4}(S^{(s-2,s-1,s)}),\dots,\allowbreak\Gamma^{s-4}(S^{(1,s-1,s)}))$ and $H_{1,2} \colonequals  (\Gamma^{s-4}(S^{(1,s-1,s)}),\dots,\Gamma^{s-4}(S^{(1,2,s)}))$.
By setting $G_{1,1} \colonequals (\Gamma^{s-4}(S^{(s-2,s-1,s)})) \cdot H_{1,1}$ and $G_{1,2} \colonequals H_{1,2} \cdot (\Gamma^{s-4}(S^{(1,2,s)}))$, we obtain that the $i$th term of $H_1$ is $\gamma(X,Y)$, where $X$ and $Y$ are the $i$th terms of $G_{1,1}$ and $G_{1,2}$, respectively.
We similarly proceed with $H_2$ and we let $H_{2,1} \colonequals H_{1,2}$ and $H_{2,2} \colonequals  (\Gamma^{s-4}(S^{(1,2,s)}),\dots,\allowbreak\Gamma^{s-4}(S^{(1,2,3)}))$.
Setting $G_{2,1} \colonequals (\Gamma^{s-4}(S^{(1,s-1,s)})\cdot H_{2,1}$ and $G_{2,2} \colonequals  H_{2,2} \cdot (\Gamma^{s-4}(S^{(1,2,3)}))$, we get that the $i$th term of $H_2$ is $\gamma(X,Y)$, where $X$ and $Y$ are the $i$th terms of $G_{2,1}$ and $G_{2,2}$, respectively; see Figure~\ref{fig-example} for an example.

\begin{figure}[ht]
\centering
\includegraphics[scale=0.95]{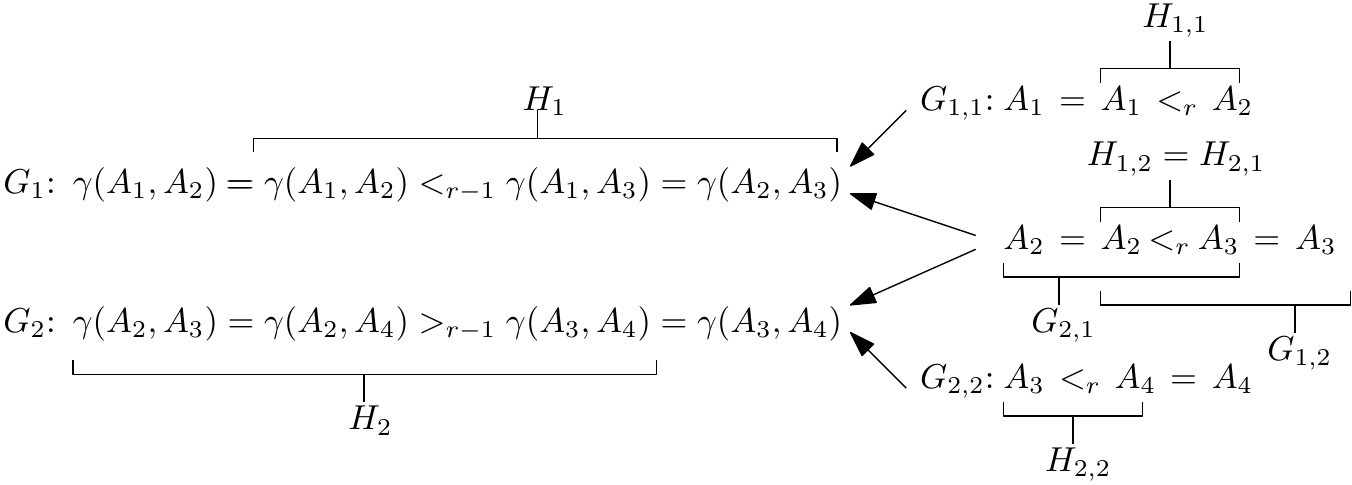}
\caption{Example of the sequences used in the induction step for $s=4$.
Here, we have profiles $p(G_1)=(=,\leq,=)$ and $p(G_2)=(=,\geq,=)$.
We set $p(H)=(=,\geq,=)$.}
\label{fig-example}
\end{figure}

We now define a profile $p(H)$ and, as our induction step, we later prove that it is a profile of~$H$.
In fact, we prove a stronger statement by additionally showing that if $p(H_1)$ and $p(H_2)$ have the same parity then either $p(H) = \overline{p(G_1)}$ or $p(H) = p(G_2)$ and also $i_1(G_1) \geq i_2(G_2)$, while if $p(H_1)$ and $p(H_2)$ have distinct parity then $i_1(G_1) \geq i_1(G_2)$ and $i_2(G_1) \geq i_2(G_2)$; see Figure~\ref{fig-profiles}.

\begin{figure}[ht]
\centering
\includegraphics{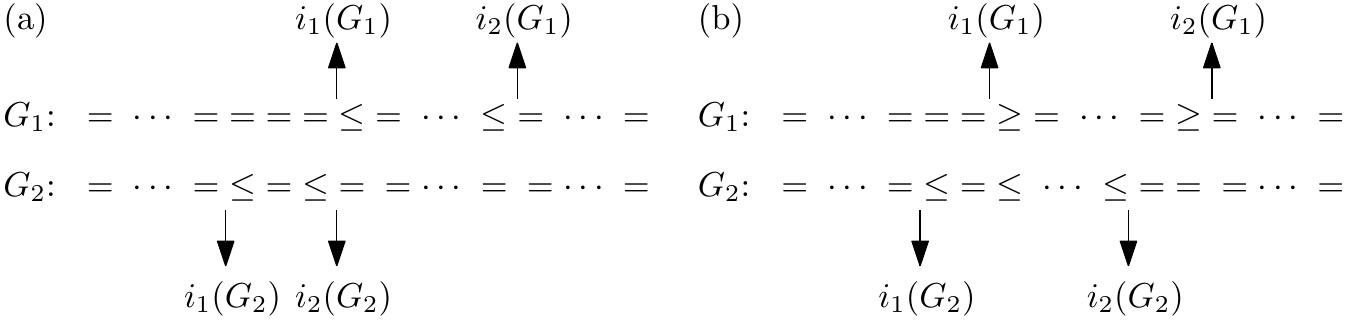}
\caption{Example of the inequalities $i_1(G_1) \geq i_2(G_2)$ in the case of the same parity of the profiles $p(H_1)$ and $p(H_2)$ (part~(a)) and $i_1(G_1) \geq i_1(G_2)$ and $i_2(G_1) \geq i_2(G_2)$ in the case when $p(H_1)$ and $p(H_2)$ have distinct parity (part~(b)).}
\label{fig-profiles}
\end{figure}

For every $j \in [s-1]$, we let the $j$th term of a profile $\overline{p(G_1)} \circ p(G_2)$ be $=$ if the $j$th terms of both $\overline{p(G_1)}$ and $p(G_2)$ are equalities and we let the $j$th term of $\overline{p(G_1)} \circ p(G_2)$ be $\leq$ if the $j$th term of $\overline{p(G_1)}$ or of $p(G_2)$ is $\leq$.
Similarly, we let the $j$th term of $\overline{p(G_1)} \circ p(G_2)$ be $\geq$ if the $j$th term of $\overline{p(G_1)}$ or of $p(G_2)$ is $\geq$.
Observe that if each of the profiles $p(H_1)$ and $p(H_2)$ is odd or even, then there is no $i$ with $1 \leq i \leq s-1$ such that the $i$th term of $\overline{p(G_1)}$ is $\leq$ while the $i$th term of $p(G_2)$ is $\geq$, or vice versa.
Thus $\overline{p(G_1)} \circ p(G_2)$ is correctly defined under this assumption.
If $p(H_1)$ and $p(H_2)$ have distinct parity, we let $p(H)$ be the profile $\overline{p(G_1)} \circ p(G_2)$.
If $p(H_1)$ and $p(H_2)$ have the same parity, we let $p(H)$ be the profile $\overline{p(G_1)}$ if $\Gamma^{s-2}(S^{(s)}) \prec_{r-s+2} \Gamma^{s-2}(S^{(1)})$ and the profile $p(G_2)$ if if $\Gamma^{s-2}(S^{(1)}) \prec_{r-s+2} \Gamma^{s-2}(S^{(s)})$.

Recall that, as our induction step, we prove that $p(H)$ is a profile of $H$ and that $i_1(G_1) \geq i_1(G_2)$ and $i_2(G_1) \geq i_2(G_2)$ if $p(H_1)$ and $p(H_2)$ have distinct parity and $i_1(G_1) \geq i_2(G_2)$ if $p(H_1)$ and $p(H_2)$ have the same parity.
We already observed that this statement is true for $s=3$.
Note that it follows from the induction hypothesis that the parity of $p(H)$ is the same as the parity of $\overline{p(G_1)}$ or of $p(G_2)$.
In particular, the profile $p(H)$ of $H$ is odd or even, which gives the statement of the lemma.

By the induction hypothesis, for every $i \in \{1,2\}$, the profile $p(H_i)$ is a profile of $H_i$ and $i_1(G_{i,1}) \geq i_1(G_{i,2})$ and $i_2(G_{i,1}) \geq i_2(G_{i,2})$ if $p(H_{i,1})$ and $p(H_{i,2})$ have distinct parity and $i_1(G_{i,1}) \geq i_2(G_{i,2})$ if $p(H_{i,1})$ and $p(H_{i,2})$ have the same parity.
In the latter case, we also know that $p(H_i)\in\{\overline{p(G_{i,1})}, p(G_{i,2})\}$.

Assume first that $p(H_1)$ and $p(H_2)$ have the same parity. 
We show that $i_1(G_1) \geq i_2(G_2)$ by distinguishing some cases.
First, we consider the case when both $p(H_1)$ and $p(H_2)$ are odd, the other one will be symmetric.
Using the definition of $G_{1,2}$ and $G_{2,1}$, the fact that $H_{1,2} = H_{2,1}$, and the fact that $p(H_{1,2}) = p(H_{2,1})$ contain at least one term which is not $=$, we obtain $i_j(G_{1,2}) = i_j(G_{2,1})-1$ for every $j \in \{1,2\}$.
\begin{enumerate}
\item We start with the cases when at least one of the following situations occurs, either $p(H_1) \notin \{\overline{p(G_{1,1})},p(G_{1,2})\}$ or $p(H_2) \notin \{\overline{p(G_{2,1})}, \allowbreak p(G_{2,2})\}$ .
Note that, by the definition of $p(H_i)$ for $i \in \{1,2\}$, if $p(H_i) \notin \{\overline{p(G_{i,1})},\allowbreak p(G_{i,2})\}$, then $p(H_i) = \overline{p(G_{i,1})} \circ p(G_{i,2})$ and the profiles $p(H_{i,1})$ and $p(H_{i,2})$ have distinct parity.

\begin{enumerate}
\item If $p(H_1) \notin \{\overline{p(G_{1,1})},p(G_{1,2})\}$ and $p(H_2) \in \{\overline{p(G_{2,1})}, \allowbreak p(G_{2,2})\}$, then $p(H_{1,1})$ is even and $p(H_{1,2})=p(H_{2,1})$ and $p(H_{2,2})$ are odd.
Since $p(H_{1,1})$ and $p(H_{1,2})$ have distinct parity, we get $i_1(G_{1,1}) \geq i_1(G_{1,2})$ and $i_2(G_{1,1}) \geq i_2(G_{1,2})$.
Since $p(H_{2,1})$ and $p(H_{2,2})$ have the same parity, we get $i_1(G_{2,1}) \geq i_2(G_{2,2})$.
From $p(H_1) = \overline{p(G_{1,1})} \circ p(G_{1,2})$ and $i_1(G_{1,1}) \geq i_1(G_{1,2})$, we get $i_1(G_1) = i_1(G_{1,2})+1$ by the definition of $G_1$.
Since $p(H_2)$ is odd, we have $p(H_2) = p(G_{2,2})$.
Thus $i_2(G_2) = i_2(G_{2,2})$.
Altogether, it follows from $i_1(G_{2,1}) \geq i_2(G_{2,2})$ and $i_1(G_{1,2}) = i_1(G_{2,1})-1$ that 
\[i_1(G_1) = i_1(G_{1,2})+1 = i_1(G_{2,1}) \geq i_2(G_{2,2})= i_2(G_2).\]

\item If $p(H_1) \in \{\overline{p(G_{1,1})},p(G_{1,2})\}$ and $p(H_2) \notin \{\overline{p(G_{2,1})}, \allowbreak p(G_{2,2})\}$, then $p(H_{1,1})$ and $p(H_{1,2})=p(H_{2,1})$ are even and $p(H_{2,2})$ is odd.
Since $p(H_{1,1})$ and $p(H_{1,2})$ have the same parity, we get $i_1(G_{1,1}) \geq i_2(G_{1,2})$.
Since $p(H_{2,1})$ and $p(H_{2,2})$ have distinct parity, we get $i_1(G_{2,1}) \geq i_1(G_{2,2})$ and $i_2(G_{2,1}) \geq i_2(G_{2,2})$.
From $p(H_2) = \overline{p(G_{2,1})} \circ p(G_{2,2})$ and $i_2(G_{2,1}) \geq i_2(G_{2,2})$, we get $i_2(G_2) = i_2(G_{2,1})$.
Since $p(H_1)$ is odd, we have $p(H_1) = \overline{p(G_{1,1})}$.
Thus $i_1(G_1) = i_1(G_{1,1})+1$ by the definition of $G_1$.
Altogether, it follows from $i_1(G_{1,1}) \geq i_2(G_{1,2})$ and $i_1(G_{1,2}) = i_1(G_{2,1})-1$ that 
\[i_1(G_1) = i_1(G_{1,1})+1 \geq i_2(G_{1,2}) + 1 = i_2(G_{2,1}) = i_2(G_2).\]

\item If $p(H_1) \notin \{\overline{p(G_{1,1})},p(G_{1,2})\}$ and $p(H_2) \notin \{\overline{p(G_{2,1})}, \allowbreak p(G_{2,2})\}$, then $p(H_{1,1})$ and $p(H_{1,2})$ have distinct parity and also $p(H_{2,1})$ and $p(H_{2,2})$ have distinct parity.
This, however, implies that either $p(H_1)$ or $p(H_2)$ is even, which is impossible.
\end{enumerate}

\item Thus now we are left with the cases $p(H_1) \in \{\overline{p(G_{1,1})},p(G_{1,2})\}$ and $p(H_2) \in \{\overline{p(G_{2,1})},\allowbreak p(G_{2,2})\}$.
We deal with all four cases.
\begin{enumerate}
\item If $p(H_1)=\overline{p(G_{1,1})}$ and $p(H_2)=p(G_{2,2})$, then $p(H_{1,1})$ is even and $p(H_{2,2})$ is odd and we have $i_1(G_1) = i_1(G_{1,1})+1$ and $i_2(G_2) = i_2(G_{2,2})$. 
If the parity of $p(H_{1,2})=p(H_{2,1})$ is odd, then $p(H_{1,1})$ and $p(H_{1,2})$ have distinct parity and $p(H_{2,1})$ and $p(H_{2,2})$ have the same parity.
It follows that $i_1(G_{1,1}) \geq i_1(G_{1,2})$ and $i_1(G_{2,1}) \geq i_2(G_{2,2})$.
Using $i_1(G_{1,2})=i_1(G_{2,1})-1$, we derive
\[i_1(G_1) = i_1(G_{1,1})+1 \geq i_1(G_{1,2})+1 = i_1(G_{2,1}) \geq i_2(G_{2,2}) = i_2(G_2).\]
If the parity of $p(H_{1,2})=p(H_{2,1})$ is even, then $p(H_{1,1})$ and $p(H_{1,2})$ have the same parity, while $p(H_{2,1})$ and $p(H_{2,2})$ have distinct parity.
This implies $i_1(G_{1,1}) \geq i_2(G_{1,2})$ and $i_2(G_{2,1}) \geq i_2(G_{2,2})$  and we derive 
\[i_1(G_1) = i_1(G_{1,1})+1 \geq i_2(G_{1,2})+1 = i_2(G_{2,1}) \geq i_2(G_{2,2}) = i_2(G_2).\]

\item Assume that $p(H_1)=\overline{p(G_{1,1})}$ and $p(H_2) = \overline{p(G_{2,1})}$.
Then $p(H_{1,1})$ and $p(H_{1,2})=p(H_{2,1})$ are both even. 
It also follows that $i_1(G_1) = i_1(G_{1,1})+1$ and $i_2(G_2) = i_2(G_{2,1})$.
Since the profiles $p(H_{1,1})$ and $p(H_{1,2})$ have the same parity, we have $i_1(G_{1,1}) \geq i_2(G_{1,2})$, which gives
\[i_1(G_1) = i_1(G_{1,1})+1 \geq i_2(G_{1,2})+1 = i_2(G_{2,1}) = i_2(G_2).\]

\item If $p(H_1)=p(G_{1,2})$ and $p(H_2) = p(G_{2,2})$, then both $p(H_{2,1})=p(H_{1,2})$ and $p(H_{2,2})$ are odd. 
We also have $i_1(G_1) = i_1(G_{1,2})+1$ and $i_2(G_2) = i_2(G_{2,2})$.
It follows that $p(H_{2,1})$ and $p(H_{2,2})$ have the same parity, which gives $i_1(G_{2,1}) \geq i_2(G_{2,2})$.
This implies
\[i_1(G_1) = i_1(G_{1,2})+1 =i_1(G_{2,1}) \geq i_2(G_{2,2}) = i_2(G_2).\]

\item We cannot have $p(H_1)=p(G_{1,2})$ and $p(H_2)=\overline{p(G_{2,1})}$, as otherwise $p(H_{1,2})$ and $p(H_{2,1})$ have distinct parity, which is impossible, as $p(H_{1,2})=p(H_{2,1})$. 
\end{enumerate}
\end{enumerate}

Altogether, if $p(H_1)$ and $p(H_2)$ are odd, we have $i_1(G_1) \geq i_2(G_2)$.
Note that in the above case analysis, we only rely on the facts that the parity of two profiles is the same or different, we do not use the actual parity.
Thus, by symmetry, the inequality $i_1(G_1) \geq i_2(G_2)$ holds if both $p(H_1)$ and $p(H_2)$ are even. 

Now, we use the fact $i_1(G_1) \geq i_2(G_2)$ to show that $p(H)$ is a profile of $H$ and $p(H)=\overline{p(G_1)}$ if $\Gamma^{s-2}(S^{(s)}) \prec_{r-s+2} \Gamma^{s-2}(S^{(1)})$ or $p(H)=p(G_2)$ if $\Gamma^{s-2}(S^{(1)}) \prec_{r-s+2} \Gamma^{s-2}(S^{(s)})$.
Recall that we assume that $p(H_1)$ and $p(H_2)$ have the same parity.
Thus $\Gamma^{s-3}(S^{(s-1,s)}) <_{r-s+3} \Gamma^{s-3}(S^{(1,s)}) <_{r-s+3} \Gamma^{s-3}(S^{(1,2)}))$ or $\Gamma^{s-3}(S^{(s-1,s)}) >_{r-s+3} \Gamma^{s-3}(S^{(1,s)}) >_{r-s+3} \Gamma^{s-3}(S^{(1,2)}))$.
This implies that the first term $\Gamma^{s-2}(S^{(s)}) = \gamma(\Gamma^{s-3}(S^{(s-1,s)}),\Gamma^{s-3}(S^{(1,s)}))$ of $H$ and the last term $\Gamma^{s-2}(S^{(1)}) = \gamma(\Gamma^{s-3}(S^{(1,s)}),\Gamma^{s-3}(S^{(1,2)}))$ of $H$ have the same type and, assuming $s \leq r$, they are not equivalent.
Thus we either have $\Gamma^{s-2}(S^{(s)}) \prec_{r-s+2} \Gamma^{s-2}(S^{(1)})$ or $\Gamma^{s-2}(S^{(1)}) \prec_{r-s+2} \Gamma^{s-2}(S^{(s)})$.
In the first case, Lemma~\ref{lem-deletion} with the parameters $A \colonequals \Gamma^{s-3}(S^{(s-1,s)}), B \colonequals \Gamma^{s-3}(S^{(1,s)}), C \colonequals \Gamma^{s-3}(S^{(1,2)})$ implies $\gamma(\Gamma^{s-3}(S^{(s-1,s)}),\Gamma^{s-3}(S^{(1,2)})) = \Gamma^{s-2}(S^{(s)})$ and in the second case, the lemma with the same parameters gives $\gamma(\Gamma^{s-3}(S^{(s-1,s)}),\allowbreak\Gamma^{s-3}(S^{(1,2)})) = \Gamma^{s-2}(S^{(1)})$.
For $s=r+1$, the terms of~$H$ lie in $F_1(n)$ and Lemma~\ref{lem-deletion} gives $\Gamma^{s-2}(S^{(1)})=\gamma(\Gamma^{s-3}(S^{(s-1,s)}),\Gamma^{s-3}(S^{(1,2)})) = \Gamma^{s-2}(S^{(s)})$ immediately.
 
We know that the term $\gamma(\Gamma^{s-3}(S^{(s-1,s)}),\allowbreak\Gamma^{s-3}(S^{(1,2)}))$ equals the first term $\Gamma^{s-2}(S^{(s)})$ of $H$ if $\Gamma^{s-2}(S^{(s)}) \prec_{r-s+2} \Gamma^{s-2}(S^{(1)})$ and to the last term $\Gamma^{s-2}(S^{(1)})$ of $H$ otherwise.
We assume without loss of generality that $\Gamma^{s-2}(S^{(s)}) \prec_{r-s+2} \Gamma^{s-2}(S^{(1)})$, as the other case is symmetric.
For $j=i_1(G_1)$, the inequality $i_1(G_1) \geq i_2(G_2)$ implies that the $j$th term of $H$ equals $\gamma(\Gamma^{s-3}(S^{(s-1,s)}),\Gamma^{s-3}(S^{(1,2)}))=\Gamma^{s-2}(S^{(s)})$.
For every $i$ with $i \leq j$, the $i$th term of $H$ is obtained by applying $\gamma$ to the first term of $G_1$ and the $i$th term of~$G_2$.
Since $G_2$ is either non-decreasing or non-increasing in $\leq_{r-s+3}$, Lemma~\ref{lem-replacement} implies that all the first $j$ terms of $H$ are monotone in $\leq_{r-s+2}$.
Thus, since the first term of $H$ and the $j$th term of $H$ are both equal to $\Gamma^{s-2}(S^{(s)})$, we get that all the first $j$ terms of $H$ are equal.
Since $j=i_1(G_1) \geq i_2(G_2)$, for every $i$ with $i > j$, the  $i$th term of $H$ is obtained by applying $\gamma$ to the $i$th term of $G_1$ and the last term of $G_2$.
Together with the previous fact, Lemma~\ref{lem-replacement} implies that $p(H)=\overline{p(G_1)}$ and it is a profile of $H$.

For the rest of the proof we assume that the profiles $p(H_1)$ and $p(H_2)$ have distinct parity. 
For $i \in [s-1]$, let $(p_i,q_i)$ be the pair consisting of the $i$th term $p_i$ of $p(G_1)$ and the $i$th term $q_i$ of $p(G_2)$.
It follows from the definition of $G_1$ and $G_2$ that $(p_i,q_i) \in \{(=,=),(\leq,=),(=,\geq),\allowbreak(\leq,\geq)\}$ if $p(H_1)$ is odd and $p(H_2)$ is even and that  $(p_i,q_i) \in \{(=,=),(\geq,=),(=,\leq),(\geq,\leq)\}$ if $p(H_1)$ is even and $p(H_2)$ is odd.
Thus, by Lemma~\ref{lem-replacement} and by the fact that the $i$th term of $H$ is obtained by applying $\gamma$ to the $i$th terms of $G_1$ and $G_2$ for each $i \in [s]$, the profile $p(H)$ is odd or even and it is a profile of $H$.
For example, in the case $(p_i,q_i) = (\leq,=)$, we apply part~(i) of Lemma~\ref{lem-replacement} with $A \colonequals i\text{th}$ term of $G_1$, $A' \colonequals (i+1)\text{st}$ term of $G_1$, and $B \colonequals i\text{th}$ term of $G_2$, which is also the $(i+1)\text{st}$ term of $G_2$.
Note that in the cases $(p_i,q_i) \in \{(\leq,\geq),(\geq,\leq)\}$ we apply Lemma~\ref{lem-replacement} twice.

It remains to show that $i_1(G_1) \geq i_1(G_2)$ and $i_2(G_1) \geq i_2(G_2)$.
Let $j \in \{1,2\}$.
Since $p(H_1) \in \{\overline{p(G_{1,1})},p(G_{1,2}),\overline{p(G_{1,1})} \circ p(G_{1,2})\}$, we have $i_j(H_1) \in \{i_j(G_{1,1}),i_j(G_{1,2})\}$.
Similarly, $p(H_2) \in \{\overline{p(G_{2,1})},p(G_{2,2}),\overline{p(G_{2,1})} \circ p(G_{2,2})\}$ and thus $i_j(H_2) \in \{i_j(G_{2,1}),i_j(G_{2,2})\}$.
Since $p(H_{1,2})=p(H_{2,1})$, it follows from the definition of $G_{1,2}$ and $G_{2,1}$ that $i_j(G_{1,2})+1=i_j(G_{2,1})$.
We recall that $i_2(G_{k,l}) \geq i_1(G_{k,l})$ for all $k,l \in \{1,2\}$.
Thus the induction hypothesis gives $i_j(G_{1,1}) \geq i_j(G_{1,2})$ and $i_j(G_{2,1}) \geq i_j(G_{2,2})$\footnote{Here, we are considering both parity cases, namely distinct or same parity of $p(H_{i,1})$ and $p(H_{i,2})$, at the same time.}.
Altogether, we obtain $i_j(H_1) \geq i_j(G_{2,1})-1$ and $i_j(H_2) \leq i_j(G_{2,1})$.
It follows from the definition of $G_1$ and $G_2$ that $i_j(G_1) = i_j(H_1)+1$ and $i_j(G_2) = i_j(H_2)$.
This implies $i_j(G_1) \geq i_j(G_2)$.
\end{proof}

Lemma~\ref{lem-Monotone} is sufficient to guarantee the monotonicity property for $c_r$.

\begin{corollary}
\label{cor-monotonicityProperty}
For every integer $r$ with $r \geq 3$, the coloring $c_r$ is $r$-monotone.
\end{corollary}
\begin{proof}
For a sequence $S\colonequals(A_1,\dots,A_{r+1})$ of sets from $F_r(n)$ with $A_1 <_r \dots <_r A_{r+1}$, we show that there is at most one change of a sign in the sequence $(c_r(S^{(r+1)}),\dots,\allowbreak c_r(S^{(1)}))$.
By Lemma~\ref{lem-Monotone} applied for $s \colonequals r+1$, the sequence $(\Gamma^{r-1}(S^{(r+1)}),\dots,\allowbreak\Gamma^{r-1}(S^{(1)}))$  has odd or even profile and, in particular, this sequence is monotone in $\leq_1$.
The rest follows from the fact that $c_r(S^{(i)})=\Gamma^{r-1}(S^{(i)})$ for every $i \in [r+1]$.
\end{proof}

Lemma~\ref{lem-noLongPath} and Corollary~\ref{cor-monotonicityProperty} together give the statement of Theorem~\ref{thm-lowerBound}.

\paragraph{Comparison with the construction by Moshkovitz and Shapira}
For positive integers $n$ and $r \geq 3$, Mosh\-kovitz and Shapira~\cite{moshSha14} constructed colorings $c'_r$ of $\mathcal{K}^r_N$ with $N \geq \tow_{r-1}(\Omega(n))$ such that there is no monochromatic copy of $\mathcal{P}^r_n$ in $c'_r$.
However, for $r \geq 4$, their coloring $c'_r$ is not transitive.
The construction of our coloring $c_r$ is inspired by their approach and uses similar ideas.
However, there are some differences.
First of all, the coloring $c'_r$ is defined on a larger vertex set formed by \emph{line partitions of order $r$}, ordered by the lexicographic order $\lessdot_r$, while the vertex set on which $c_r$ is defined can be regarded as a proper subset of the vertex set for $c'_r$.
Second, the function $\gamma$ in the definition of $c_r$ differs from a function $\delta(A,B)$ that is used in the definition of $c'_r$ and that returns the smallest element of $B \setminus A$ in $\lessdot_{r-1}$.
Our function $\gamma$ is defined very similarly, but it uses the ordering $\prec_{r-1}$ instead.

\section{Proof of Theorem~\ref{thm-signotopesCount}}
\label{sec-signotopeCount}

In this section, we prove Theorem~\ref{thm-signotopesCount} by showing that the number of $r$-monotone colorings of~$\mathcal{K}^r_n$ is of order $2^{n^{r-1}/r^{\Theta(r)}}$ for $r \geq 3$ and $n \geq r$.
We first derive the lower bound in Subsection~\ref{subsec-countingLowerBound} and then, in Subsection~\ref{subsec-countingUpperBound}, we prove the upper bound.

\subsection{A lower bound on the number of monotone colorings}
\label{subsec-countingLowerBound}

Here we provide a lower bound $2^{n^{r-1}/r^{O(r)}}$ on the number of $r$-monotone colorings of $\mathcal{K}^r_n$ with $r \geq 3$ and $n \geq r$.
The construction is inspired by the method used by Matou\v{s}ek~\cite{mat02} to show that there are $2^{\Omega(n^2)}$ simple arrangements of $n$ pseudolines. 

First, we introduce some definitions.
A \emph{composition} of a positive integer $m$ into $k$ parts, $k \in \mathbb{N}$, is an ordered $k$-tuple $(p_1,\dots,p_k)$ of positive integers with $p_1+\dots+p_k = m$.
It is well-known and easy to show that the number of compositions of $m$ into $k$ parts is exactly $\binom{m-1}{k-1}$.
In particular, the total number of compositions of $m$ is $\sum_{i=1}^m \binom{m-1}{i-1} = 2^{m-1}$.

Let $r$ and $k$ be integers with $r \geq 3$ and $1 \leq k \leq r$.
Let $\sigma=(p_1,\dots,p_k)$ be a composition of~$r$ into $k$ parts.
The \emph{reduction step} on $\sigma$ maps $\sigma$ to the composition $(p_1,\dots,p_k-1)$ if $p_k > 1$ or to the composition $(p_1, \dots, p_{k-1})$ if $p_k = 1$.
We say that a composition $\sigma'$ is the \emph{reduction} of $\sigma$ if $\sigma'$ is a composition of one of the forms $(1,\dots,1,2)$ or $(p,1)$, for some $p>1$, and is obtained from $\sigma$ by a sequence of reduction steps.
Note that $\sigma$ has a reduction if and only if $\sigma \neq (1,\dots,1)$ and $\sigma \neq (r)$.
Moreover, the reduction, if it exists, is unique.

We now recursively define the sign of a composition $\sigma$ of $r$ using the sign of its reduction.
This is carried out by induction on $r$.
If $r=3$, then $\sigma$ is \emph{negative} if $\sigma=(1,2)$ and $\sigma$ is \emph{positive} if $\sigma = (2,1)$.
For $r>3$, we say that $\sigma$ is \emph{negative} if it satisfies one of the following three conditions: the reduction of $\sigma$ is negative, $\sigma=(1,\dots,1,2)$ and $r$ is odd, or $\sigma=(r-1,1)$ and $r$ is even. 
Similarly, we say that $\sigma$ is \emph{positive} if it satisfies one of the following three conditions: the reduction of $\sigma$ is positive, $\sigma=(1,\dots,1,2)$ and $r$ is even, or $\sigma=(r-1,1)$ and $r$ is odd. 
The notion of negative and positive integer compositions is illustrated in Figure~\ref{fig-compositions}.
Note that, for every $r \geq 3$, the only two compositions of $r$ that are not negative nor positive are $(1,\dots,1)$ and $(r)$.

\begin{figure}[ht]
\centering
\includegraphics{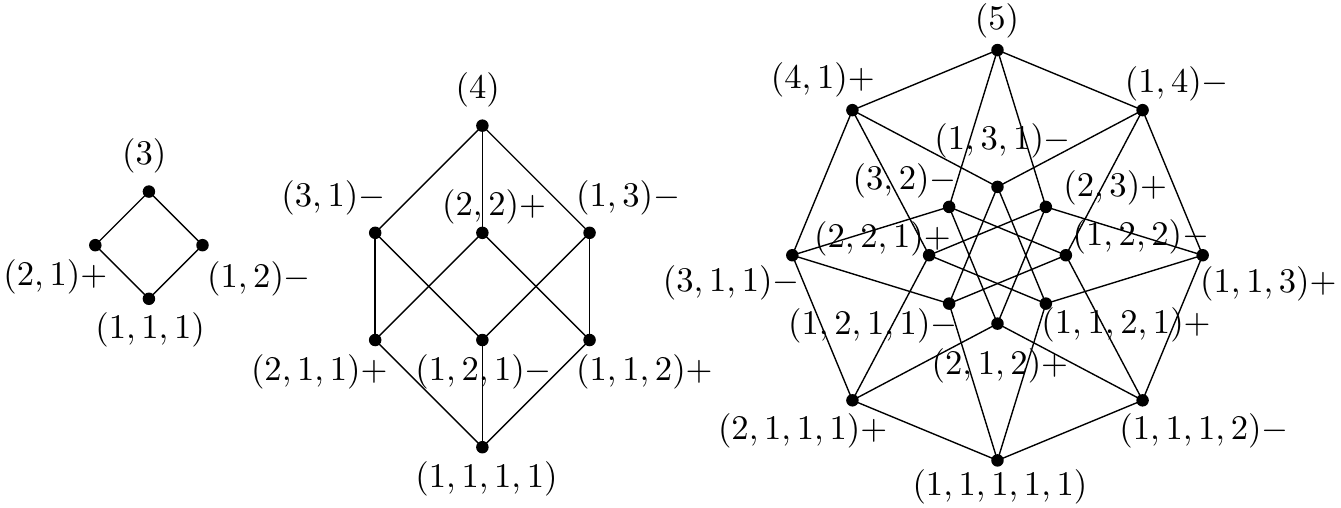}
\caption{Examples of negative and positive compositions of $r\in \{3,4,5\}$.
The compositions of $r$ are illustrated as elements of the partially ordered set~$(2^{[r-1]},\subseteq)$ and the sign $-$ or $+$ next to a composition $\sigma$ denotes whether $\sigma$ is negative or positive, respectively.}
\label{fig-compositions}
\end{figure}

Let $r$ and $h$ be positive integers with $r \geq 3$.
We set $n \colonequals r^h$ and $m \colonequals n/r = r^{h-1}$.
We now present a construction of a 3-coloring $c_{r,h}$ of $\mathcal{K}^r_n$ with colors $\{-,0,+\}$ such that every 2-coloring that is obtained by replacing each occurrence of the color $0$ with either $-$ or $+$ is $r$-monotone.
The construction is carried out recursively starting with the case $h=1$, in which $n=r$ and $c_{r,1}$ is the coloring that assigns the color $0$ to the only edge $[r]$ of~$\mathcal{K}^r_n$.

For $h \geq 2$, we let $V_i \colonequals \{(i-1)m+1,\dots,im\}$ for every $i \in [r]$ and we let $[n]$ be the vertex set of $\mathcal{K}_n^r$.
Note that the sets $V_1,\dots,V_r$ partition $[n]$ and form consecutive intervals of size $m$ in the ordering $<$ on $[n]$. 

We define the 3-coloring $c_{r,h}$ of $\mathcal{K}^r_n$ on $[n]$ as follows.
Let $e=\{v_1,\dots,v_r\} \in \binom{[n]}{r}$ be an edge of $\mathcal{K}^r_n$.
The sets $V_1,\dots,V_r$ partition $e$ into nonempty sets $e_1,\dots,e_k$, for some $k \in [r]$, that are consecutive in $<$.
We let $p_i$ be the size of $e_i$ for every $i \in [k]$ and we use $\sigma$ to denote the composition $(p_1,\dots,p_k)$ of $r$.
We choose $c_{r,h}(e) \colonequals -$ if $\sigma$ is negative and $c_{r,h}(e)\colonequals+$ if $\sigma$ is positive.
It remains to assign the color $c_{r,h}(e)$ to edges $e$ for which $\sigma$ is not negative nor positive, that is, to edges $e$ for which either $\sigma = (r)$ or $\sigma=(1,\dots,1)$.
If $\sigma=(r)$, then $e \subseteq V_i$ for some $i \in [r]$ and, in particular, $\{v_1-(i-1)m,\dots,v_r-(i-1)m\} \subseteq [m]$.
We then use the coloring $c_{r,h-1}$ from the previous step of the construction and we let $c_{r,h}(e) \colonequals c_{r,h-1}(\{v_1-(i-1)m,\dots,v_r-(i-1)m\})$.
If $\sigma = (1,\dots,1)$, then each $v_i$ lies in the set $V_i$.
In this case, we use $v'_i$ to denote the integer $v_i - (i-1)m$ from $[m]$ and we let
\[
c_{r,h}(e) \colonequals
\begin{cases}
- & \text{if } \sum\limits_{\substack{i \in [r]\\ i\text{ even}}}v'_i < \sum\limits_{\substack{i \in [r]\\ i\text{ odd}}}v'_i,\\
0 & \text{if } \sum\limits_{\substack{i \in [r]\\ i\text{ even}}}v'_i  = \sum\limits_{\substack{i \in [r]\\ i\text{ odd}}}v'_i,\\
+ & \text{if } \sum\limits_{\substack{i \in [r]\\ i\text{ even}}}v'_i >  \sum\limits_{\substack{i \in [r]\\ i\text{ odd}}}v'_i.
\end{cases}
\]
This finishes the construction of $c_{r,h}$.
We show that no matter how we replace zeros with $-$ or $+$ signs in $c_{r,h}$, the resulting coloring is $r$-monotone.

\begin{lemma}
\label{lem-countSignotope}
For $h \geq 1$ and $r \geq 3$, let $c$ be an arbitrary 2-coloring of $\mathcal{K}^r_n$ that is obtained from $c_{r,h}$ by replacing  each occurrence of $0$ with $-$ or $+$.
Then $c$ is an $r$-monotone coloring of~$\mathcal{K}^r_n$.
\end{lemma}
\begin{proof}
We prove the statement by induction on $h$.
For $h=1$, the statement is trivial as $n=r$ and there is only a single edge in $\mathcal{K}^r_r$.

Now, assume that $h \geq 2$.
We further assume that the statement is true for $h-1$.
Let $F = \{v_1,\dots,v_{r+1}\} \subseteq [n]$ be an $(r+1)$-tuple of vertices of $\mathcal{K}^r_n$ with $v_1 < \dots < v_{r+1}$ and let $j_1<\dots<j_k$ be indices with $F \cap V_{j_i} \neq\emptyset$.
We let $\sigma=(p_1,\dots,p_k)$, $k \in [r]$, be the composition of $r+1$, where $p_i=|F\cap V_{j_i}|$ for every $i \in [k]$.
For every $i \in [r+1]$, we let $e_i$ be the edge $F \setminus \{v_i\}$.
Similarly as before, for every $i \in [r+1]$, the partitioning of each edge $e_i$ by $V_1,\dots,V_r$ determines a composition $\sigma_i$ of $r$.
Note that each $\sigma_i$ can be obtained from $\sigma$ by decreasing $p_j$ by 1 if $p_j>1$ or by removing $p_j$ if $p_j=1$, where $j$ is a number from~$[k]$ such that $\sum_{l=1}^{j-1} p_l < i$ and $\sum_{l=1}^j p_l \geq i$.

We show that $c$ is $r$-monotone  by proving that there is at most one change of a sign in the sequence $S_F\colonequals (c(e_1),\dots,c(e_{r+1}))$.
Since there are only $r$ sets $V_1,\dots,V_r$ in the partition of $[n]$, we cannot have $\sigma = (1,\dots,1)$.
If $\sigma = (r+1)$, then $F \subseteq V_i$ for some $i \in [r]$ and the statement follows from the induction hypothesis for $h-1$.
Thus we can assume that $\sigma$ is positive or negative.

We first deal with the case $\sigma = (1,\dots,1,2,1,\dots,1)$, that is, $p_j=2$ for some $j \in [r]$ and $p_i=1$ for every $i \in [r]\setminus \{j\}$.
For such a $\sigma$, we have $\sigma_j=(1,\dots,1)=\sigma_{j+1}$, every $\sigma_i$ with $i > j+1$ has the $j$th coordinate $2$ and all other $1$, and $\sigma_i$ with $i < j$ has the $(j-1)$st coordinate $2$ and all other $1$.
We show that if $\sigma_i$ has the value $2$ on an odd coordinate, then $c_{r,h}(e_i)=+$.
This is because we can perform reduction steps until we reach the reduction $(1,\dots,1,2)$ of $\sigma_i$.
This reduction has an odd number of parts, which implies that it is a composition of an even number and thus the reduction of $\sigma_i$ is positive.
By the definition of $c_{r,h}$, we obtain $c_{r,h}(e_i)=+$.
Similarly, if the value $2$ is on an even coordinate of $\sigma_i$, then $c_{r,h}(e_i)=-$.
Altogether, we see that there are $\xi,\xi',\xi'' \in \{-,+\}$ such that $S_F=(\xi,\dots,\xi,\xi',\xi'',-\xi,\dots,-\xi)$, where $\xi'$ and $\xi''$ are on the $j$th and the $(j+1)$st coordinate, respectively.
Moreover, $\xi = +$ if $j$ is even and $\xi = -$ if $j$ is odd.
Since $v_j, v_{j+1} \in V_{i_j} = V_j$ and $v_j < v_{j+1}$, we have $v'_j < v'_{j+1}$.
Moreover, since $e_j=F \setminus \{v_j\}$, $e_{j+1}=F \setminus \{v_{j+1}\}$, the definition of $c_{r,h}$ implies that $\xi' \leq \xi''$ if $j$ is odd and $\xi' \geq \xi''$ if $j$ is even and  either $c_{r,h}(e_j)$  or $c_{r,h}(e_{j+1})$ is not $0$.
Thus there is at most one change of a sign in $S_F$.

In the rest of the proof, we assume that $\sigma$ is a negative or a positive composition of $r+1$ that is not of the form $(1,\dots,1,2,1,\dots,1)$.
Let $\sigma'$ be the reduction of $\sigma$.
We know that $\sigma'$ is a composition of some integer $r'$ with $3 \leq r' \leq r+1$ and $\sigma'=(r'-1,1)$ or $\sigma'=(1,\dots,1,2)$.

First, we consider the case where $\sigma'$ is of the form $(r'-1,1)$.
For every $i \in [k]$ with $i > r'$, the composition $\sigma_i$ has the same reduction as $\sigma$ and thus all the edges $e_i$ with $i > r'$ have the same color $\xi \in \{-,+\}$ in $c_{r,h}$.
Assume that $r'>3$.
Then every composition $\sigma_i$ with $i < r'$ has the reduction $(r'-2,1)$ and thus every edge $e_i$ with $i < r'$ has the color $-\xi$ in $c_{r,h}$.
It follows that $c$ is $r$-monotone, as $S_F=(-\xi,\dots,-\xi,\xi',\xi,\dots,\xi)$ for some $\xi' \in \{-,+\}$.
Now, assume $r' = 3$. 
Since $\sigma \neq (2,1,\dots,1)$, there is an entry in $\sigma$ of size larger than $1$ not lying on the first position and thus $\sigma_1$ and $\sigma_2$ have the same reduction of the form $(1,\dots,1,2)$.
Since $r+1 \geq 4$ and $r'=3$, there is at least one entry in $\sigma_3$ besides the first entry $r'-1=2$ and thus $\sigma_3$ has the same reduction $(r'-1,1)=(2,1)$ as any $\sigma_i$ with $i > r'=3$.
It follows that $S_F=(\xi',\xi',\xi,\dots,\xi)$ for some $\xi' \in \{-,+\}$.

Now, we consider the case $\sigma' = (1,\dots,1,2)$.
The composition $\sigma'$ is the reduction of $\sigma_i$ for every $i \in [k]$ with $i > r'$ and thus all the edges $e_i$ with $i > r'$ have the same color $\xi \in \{-,+\}$ in~$c_{r,h}$.
Since $\sigma \neq (1,\dots,1,2,1,\dots,1)$, the compositions $\sigma_{r'-1}$ and $\sigma_{r'}$ have the same reduction.
Assume $r' > 3$.
Then every $\sigma_i$ with $i \leq r'-2$ has the reduction $(1,\dots,1,2)$, which is a composition of $r'-1$.
Consequently, for every $i \leq r'-2$, the edge $e_i$ has color $-\xi$ in $c_{r,h}$.
Thus $S_F=(-\xi,\dots,-\xi,\xi',\xi',\xi,\dots,\xi)$ for some $\xi' \in \{-,+\}$.
If $r' = 3$, then $\sigma' = (1,2)$ and the reduction of $\sigma_1$ is $(p_2,1)$.
If $p_2 \geq 3$, then the compositions $\sigma_2,\dots,\sigma_{r+1}$ have the same reduction and $S_F=(\xi',\xi,\dots,\xi)$ for some $\xi' \in \{-,+\}$.
If $p_2=2$, then the reduction of $\sigma_1$ is $(2,1)$ and, since $(2,1)$ is positive and $(1,2)$ is negative, we obtain $S_F=(+,\xi',\xi',-,\dots,-)$ for some $\xi' \in \{-,+\}$.
In any case, there is at most one change of a sign in $S_F$ and $c$ is $r$-monotone. 
\end{proof}

By Lemma~\ref{lem-countSignotope}, every coloring obtained from $c_{r,h}$ is $r$-monotone.
Thus, to finish the proof of the lower bound in Theorem~\ref{thm-signotopesCount}, it suffices to estimate the number of such colorings from~below.

\begin{lemma}
\label{lem-countEstimate}
For positive integers $h$ and $r$ with $r \geq 3$, there are at least
\[2^{r^{(r-1)(h-1)-2r}}\]
colorings that can be obtained from $c_{r,h}$ by replacing each occurrence of $0$ with $-$ or~$+$.
\end{lemma}
\begin{proof}
Let $f_r(h)$ be the number of 2-colorings that can be obtained from~$c_{r,h}$ by replacing  each occurrence of color $0$ with either $-$ or $+$.
Clearly, we have $f_r(1)=2$.
For $h \geq 2$, we have $f_r(h) \geq 2^x$, where $x$ is the number of edges of color $0$ in $c_{r,h}$ that are not contained in any $V_i$.
We recall that $m = r^{h-1} \geq r$.

We estimate the number $x$ as follows.
Consider an arbitrary $(r-1)$-tuple $T=(t_1,\dots,t_{r-1})$ of numbers from $[\lceil m/2 \rceil]$ such that not all terms of $T$ are equal.
Clearly, there are $\lceil m/2\rceil^{r-1}-\lceil m/2\rceil$ such $(r-1)$-tuples.
Let $I$ and $J$ be two sets of sizes $\lceil (r-1)/2 \rceil$ and $\lfloor (r-1)/2 \rfloor$, respectively, whose union is a partition of $[r-1]$ such that $d\colonequals\sum_{i \in I} t_i - \sum_{j \in J} t_j$ is minimum and positive.
Such a partition exists, as not all terms of $T$ are equal and $|I| \geq |J|$.
We claim that $d \leq m$.

Suppose for contradiction that $d > m$.
Let $t_a$ be the largest element from $(t_i \colon i \in I)$ and let $t_b$ be the smallest element from $(t_j \colon j \in J)$.
Note that $t_a > t_b$, as $d > m$ and every element from $(t_i \colon i \in I)$ is at most $\lceil m/2 \rceil \leq m$.
Let $I'\colonequals (I \setminus \{a\}) \cup \{b\}$ and $J'\colonequals (J \setminus \{b\}) \cup \{a\}$.
The value $\sum_{i \in I'} t_i - \sum_{j \in J'} t_j$ decreases by $2(t_a-t_b)$ when compared to $\sum_{i \in I} t_i - \sum_{j \in J} t_j$.
Since $1 \leq 2(t_a - t_b) \leq 2\lceil m/2\rceil -2 \leq m$, we have $0 <\sum_{i \in I'} t_i - \sum_{j \in J'} t_j<d$, which contradicts the choice of $I$ and~$J$.

We let $E \colonequals (t_i \colon i \in I)$ and $O$ be the sequence that is obtained from $(t_j \colon j \in J)$ by adding the element $d \in [m]$.
Then $|O|=\lfloor(r-1)/2 \rfloor +1 = \lceil r/2 \rceil$, $|E|=\lceil (r-1)/2\rceil = \lfloor r/2\rfloor$, and $\sum_{o \in O}t_o = \sum_{e \in E}t_e$.
We choose $v'_{2i-1}$ to be the $i$th element of~$O$ for every $i \in [\lceil r/2\rceil]$ and $v'_{2j}$ to be the $j$th element of~$E$ for every $j \in [\lfloor r/2\rfloor]$.
Then we set $v_i \colonequals v'_i + (i-1)m \in V_i$ for each $i \in [r]$ and obtain $c_{r,h}(\{v_1,\dots,v_r\})=0$.
One sequence $(v'_1,\dots,v'_r)$ is obtained from at most $r!$ $(r-1)$-tuples $T$ with $d$ added.
Thus $x \geq (\lceil m/2\rceil ^{r-1}-\lceil m/2\rceil)/r!$.
Altogether, we have an estimate $f_r(h) \geq 2^{((m/2)^{r-1}-\lceil m/2 \rceil)/r!}$, which is at least $2^{r^{(r-1)(h-1)-{2r}}}$, as $m=r^{h-1}$.
\end{proof}

If $n=r^h$, then the bound from Lemma~\ref{lem-countEstimate} gives the lower bound $2^{n^{r-1}/r^{3r}}$ on the number of $r$-monotone colorings of $\mathcal{K}^r_n$.
For $n$ that is not a power of $r$, we have $r^{h-1} < n < r^h$ for some $h \in \mathbb{N}$ and we can use the estimate $2^{n^{r-1}/r^{4r}}$.

\subsection{An upper bound on the number of monotone colorings}
\label{subsec-countingUpperBound}

Here, using a result of Felsner and Valtr~\cite{felVal11}, we show that, for integers $r \ge 3$ and $n \geq r$, the number of $r$-monotone colorings of $\mathcal{K}^r_n$ is at most $2^{2^{r-2}n^{r-1}/(r-1)!}$.

We proceed by induction on $r$.
For $r=3$, Felsner and Valtr~\cite{felVal11} showed that the number of sign functions of simple arrangements of $n$ pseudolines is at most $2^{0.657n^2} \leq 2^{n^2}$.
By Theorem~\ref{thm-interHyperpl}, sign functions of simple arrangements of~$n$ pseudolines correspond to $3$-monotone colorings of $\mathcal{K}^3_n$ and thus the number of such monotone colorings is also at most $2^{n^2}$.
This constitutes the base case.

For the induction step, we assume $r \geq 4$.
Let $c$ be an $r$-monotone coloring of $\mathcal{K}^r_n$ with vertex set $[n]$.
For $i \in \{r,\dots,n\}$, the \emph{$i$th projection of $\mathcal{K}^r_n$} is the function $p_i$ that maps an edge $\{v_1,\dots,v_{r-1},i\}$ of  $\mathcal{K}^r_n$ with $v_1 < \dots < v_{r-1} < i$ to $\{v_1,\dots,v_{r-1}\}$.
The image of $\mathcal{K}^r_n$ via~$p_i$ is the ordered complete $(r-1)$-uniform hypergraph $\mathcal{K}^{r-1}_{i-1}$.
Note that for every edge $e$ of $\mathcal{K}^{r-1}_{i-1}$ there is a unique edge $e'=p_i^{-1}(e)$ of $\mathcal{K}^r_n$ with $p_i(e')=e$. 
If $c$ is an $r$-monotone coloring of $\mathcal{K}^r_n$, then we use $p_i(c)$ to denote the 2-coloring of $\mathcal{K}^{r-1}_{i-1}$ obtained by coloring an edge $e$ of $\mathcal{K}^{r-1}_{i-1}$ with the color $c(p_i^{-1}(e))$.

We show that every $p_i(c)$ is an $(r-1)$-monotone coloring of $\mathcal{K}^{r-1}_{i-1}$.
Suppose for contradiction that there is an $i \in \{r,\dots,n\}$ such that $p_i(c)$ is not an $(r-1)$-monotone coloring of~$\mathcal{K}^{r-1}_{i-1}$.
Then there is an $r$-tuple $R$ of vertices from $[i-1]$ such that the sequence $S_R=(p_i(c)(R^{(r)}),\dots,p_i(c)(R^{(1)}))$ has at least two changes of a sign.
It follows from the definition of $p_i$ that, for the $(r+1)$-tuple $T=R\cup\{i\}$, we have $c(T^{(j)})=p_i(c)(R^{(j)})$ for every $j \in [r]$.
Thus the sequence $S_T=(c(T^{(r+1)}),\dots,c(T^{(1)}))$ equals to the sequence that is obtained from~$S_R$ by adding the first coordinate $c(T^{(r+1)})=c(R)$.
Then, however, there are at least two changes of a sign in $S_T$, which contradicts the assumption that $c$ is $r$-monotone.

Every $r$-monotone coloring $c$ of $\mathcal{K}^r_n$ thus yields a sequence $S_c=(p_r(c),\allowbreak\dots,p_n(c))$ of $(r-1)$-monotone colorings.
Moreover, the mapping $c \mapsto S_c$ is injective.
For every $i \in \{r,\dots,n\}$, the number of choices for $p_i(c)$ is at most $2^{2^{r-3}(i-1)^{r-2}/(r-2)!}$ by the induction hypothesis.
Altogether, the number of sequences~$S_c$, and thus also the number of $r$-monotone colorings of~$\mathcal{K}^r_n$, is at most
\[\prod_{i=r}^n 2^{2^{r-3}(i-1)^{r-2}/(r-2)!} \leq 2^{(2^{r-3}/(r-2)!)\sum_{i=1}^n  i^{r-2}} \leq 2^{2^{r-2}n^{r-1}/(r-1)!}.\]
To derive the last inequality, we used the estimate $\sum_{i=1}^n i^{r-2} \leq \frac{n^{r-1}}{r-1} + n^{r-2} \leq 2n^{r-1}/(r-1)$ for the power sum~\cite{beard96}.
This finishes the proof of the upper bound in Theorem~\ref{thm-signotopesCount}.

\subsection*{Acknowledgment}
First of all, I would like to thank the anonymous referees for carefully going through the manuscript and for very helpful suggestions that really improved the overall presentation of the paper.
I would also like to thank Attila P\'{o}r and Pavel Valtr for interesting discussions during the early stages of the research.

\bibliographystyle{plain}
\bibliography{bibliography}

\begin{thebibliography}{10}

\bibitem{bckk15}
Martin Balko, Josef Cibulka, Karel Kr{\' a}l, and Jan Kyn{\v c}l.
\newblock Ramsey numbers of ordered graphs.
\newblock {\em Electronic Notes in Discrete Mathematics}, 49:419--424, 2015.

\bibitem{barMatPor16}
Imre B{\' a}r{\' a}ny, Ji{\v r}{\' i} Matou{\v s}ek, and Attila P{\' o}r.
\newblock Curves in {$\mathbb{R}^d$} intersecting every hyperplane at most
  {$d+1$} times.
\newblock {\em J. Eur. Math. Soc. (JEMS)}, 18(11):2469--2482, 2016.

\bibitem{beard96}
Alan~F. Beardon.
\newblock Sums of powers of integers.
\newblock {\em Amer. Math. Monthly}, 103(3):201--213, 1996.

\bibitem{cfls17}
David Conlon, Jacob Fox, Choongbum Lee, and Benny Sudakov.
\newblock Ordered {R}amsey numbers.
\newblock {\em J. Combin. Theory Ser. B}, 122:353--383, 2017.

\bibitem{cfpss13}
David Conlon, Jacob Fox, J{\'a}nos Pach, Benny Sudakov, and Andrew Suk.
\newblock Ramsey-type results for semi-algebraic relations.
\newblock {\em Trans. Amer. Math. Soc.}, 366(9):5043--5065, 2014.

\bibitem{eliMat13}
Marek Eli{\' a}{\v s} and Ji{\v r}{\' i} Matou{\v s}ek.
\newblock Higher-order {E}rd{\H o}s--{S}zekeres theorems.
\newblock {\em Adv. Math.}, 244:1--15, 2013.

\bibitem{emrs14}
Marek Eli{\'a}{\v s}, Ji{\v r}{\' i} Matou{\v s}ek, Edgardo Rold{\'
  a}n-Pensado, and Zuzana Safernov{\' a}.
\newblock Lower bounds on geometric {R}amsey functions.
\newblock {\em SIAM J. Discrete Math.}, 28(4):1960--1970, 2014.

\bibitem{erdos47}
Paul Erd{\H o}s.
\newblock Some remarks on the theory of graphs.
\newblock {\em Bull. Amer. Math. Soc.}, 53:292--294, 1947.

\bibitem{erdSze35}
Paul Erd{\H o}s and George Szekeres.
\newblock A combinatorial problem in geometry.
\newblock {\em Compositio Math.}, 2:463--470, 1935.

\bibitem{felsner97}
Stefan Felsner.
\newblock On the number of arrangements of pseudolines.
\newblock {\em Discrete Comput. Geom.}, 18(3):257--267, 1997.

\bibitem{felVal11}
Stefan Felsner and Pavel Valtr.
\newblock Coding and counting arrangements of pseudolines.
\newblock {\em Discrete Comput. Geom.}, 46(3):405--416, 2011.

\bibitem{felWei01}
Stefan Felsner and Helmut Weil.
\newblock Sweeps, arrangements and signotopes.
\newblock {\em Discrete Appl. Math.}, 109(1--2):67--94, 2001.

\bibitem{fpss12}
Jacob Fox, J{\' a}nos Pach, Benny Sudakov, and Andrew Suk.
\newblock {E}rd{\H o}s--{S}zekeres-type theorems for monotone paths and convex
  bodies.
\newblock {\em Proc. Lond. Math. Soc. (3)}, 105(5):953--982, 2012.

\bibitem{goodmanPollack84}
Jacob~E. Goodman and Richard Pollack.
\newblock Semispaces of configurations, cell complexes of arrangements.
\newblock {\em J. Combin. Theory Ser. A}, 37(3):257--293, 1984.

\bibitem{knuth92}
Donald~E. Knuth.
\newblock {\em Axioms and hulls}, volume 606 of {\em Lect. Notes Comput. Sci.}
\newblock Springer-Verlag, 1992.

\bibitem{mat02}
Ji{\v r}{\' i} Matou{\v s}ek.
\newblock {\em Lectures on discrete geometry}.
\newblock Springer-Verlag, New York, 2002.

\bibitem{msw15}
Kevin~G. Milans, Derrick Stolee, and Douglas~B. West.
\newblock Ordered {R}amsey theory and track representations of graphs.
\newblock {\em J. Comb.}, 6(4):445--456, 2015.

\bibitem{miya17}
Hiroyuki Miyata.
\newblock On combinatorial properties of points and polynomial curves.
\newblock Preliminary version: \url{http://arxiv.org/abs/1703.04963}, 2017.

\bibitem{moshSha14}
Guy Moshkovitz and Asaf Shapira.
\newblock {R}amsey theory, integer partitions and a new proof of the {E}rd{\H
  o}s--{S}zekeres theorem.
\newblock {\em Adv. Math.}, 262:1107--1129, 2014.

\bibitem{suk14}
Andrew Suk.
\newblock A note on order-type homogeneous point sets.
\newblock {\em Mathematika}, 60(1):37--42, 2014.

\bibitem{zieg93}
G{\" u}nter~M. Ziegler.
\newblock Higher {B}ruhat orders and cyclic hyperplane arrangements.
\newblock {\em Topology}, 32(2):259--279, 1993.

\end{thebibliography}

\end{document}